\documentclass[]{amsart}

\usepackage{amsmath}
\usepackage{amssymb,amsfonts}
\usepackage{xypic}
\xyoption{all}

\newtheorem{theorem}{Theorem}[section]
\newtheorem{definition}[theorem]{Definition}
\newtheorem{lemma}[theorem]{Lemma}

\newtheorem{remark}[theorem]{Remark}
\newtheorem{example}[theorem]{Example}

\newtheorem{notation}[theorem]{Notation}

\newcommand{\Ob}{\mathrm{Ob}}

\newcommand{\F}{\mathfrak{F}}
\newcommand{\LL}{\mathfrak{L}}

\newcommand{\A}{\mathcal{A}}
\newcommand{\C}{\mathcal{C}}

\newcommand{\z}{\mathbf{z}}
\newcommand{\Set}{\mathbf{Set}}
\newcommand{\Bicat}{\mathbf{Bicat}}
\newcommand{\Cat}{\mathbf{Cat}}

\newcommand{\aso}{\boldsymbol{a}}
\newcommand{\br}{\boldsymbol{r}}
\newcommand{\bl}{\boldsymbol{l}}

\newcommand{\BB}{\mathrm{B}}
\newcommand{\B}{\mathcal{B}}
\newcommand{\lfunc}{\ensuremath{LaxFunc}}

\begin{document}
\title{\em Bicategorical homotopy fiber sequences}

\author{M. Calvo, A.M. Cegarra, B.A. Heredia}

\thanks{This work has been supported by DGI
of Spain, Project MTM2011-22554. Also, the first author by FPU grant FPU12-01112, and third author by
FPU grant AP2010-3521.}

\address{\newline
Departamento de \'Algebra, Facultad de Ciencias, Universidad de
Granada.
\newline 18071 Granada, Spain \newline
 mariacc@ugr.es \ acegarra@ugr.es\ baheredia@ugr.es }



\begin{abstract} Small B\'{e}nabou's  bicategories and, in particular, Mac Lane's monoidal
categories, have well-understood classifying spaces, which give
geometric meaning to their cells. This paper contains some
contributions to the study of the relationship between bicategories
and the homotopy types of their classifying spaces. Mainly,
generalizations are given of Quillen's Theorems A and B to lax
functors between bicategories.

\begin{center}{\em Dedicated to Hvedri Inassaridze on his 80th
birthday}\end{center}
\end{abstract}

\maketitle

\section{Introduction and summary}
Categorical structures have numerous applications outside  of
category theory proper as they occur naturally in many branches of
mathematics, physics and computer science. In particular,
higher-dimensional categories provide a suitable tool for the
treatment of an extensive list of issues with recognized
mathematical interest in algebraic topology, algebraic geometry,
algebraic $K$-theory, string field theory, conformal field theory
and statistical mechanics, as well as in the study of geometric
structures on low-dimensional manifolds. See the recent book {\em
Towards Higher Categories} \cite{Baez-May}, which provides a useful
background for this subject.

Like small categories \cite{quillen}, small B\'{e}nabou's
bicategories and, in particular, Mac Lane's monoidal categories, are
closely related to topological spaces through the classifying space
 construction. This assigns to each  bicategory $\B$
a CW-complex $\BB \B$ whose cells give a natural geometric meaning
to the cells of the bicategory \cite{ccg-1}. By this correspondence,
for example, bigroupoids correspond to homotopy 2-types
(CW-complexes whose $n^{\mathrm{th}}$ homotopy group at any base
point vanishes for $n\geq 3$), and arbitrary monoidal categories  to
delooping  of the classifying spaces of the underlying categories
(up to group completion).  The process of taking classifying spaces
of bicategories reveals  a way to transport categorical coherence to
homotopical coherence since the construction $\B\mapsto \BB\B$
preserves products, any lax or oplax functor between bicategories,
$F:\A\to\B$, induces a continuous map on classifying spaces $\BB
F:\BB\A\to \BB\B$,  any lax or oplax
  transformation between these,
  $\alpha:F\Rightarrow F'$, induces a homotopy between the corresponding induced maps
  $\BB\alpha:\BB F\Rightarrow \BB F'$, and any modification  between
   these, $\varphi:\alpha\Rrightarrow\beta$, a homotopy $\BB\varphi:\BB\alpha\Rrightarrow\BB\beta$
    between them. Thus, if $\A$ and $\B$ are biequivalent bicategories or
    if a homomorphism $\A\to \B$ has a biadjoint, then their associated classifying spaces
    are homotopy equivalent.

In this paper we show  the subtlety of this theory by analyzing the
homotopy fibers of the map $\BB F:\BB \A\to \BB\B$, which is induced
by a lax functor between small bicategories  $F:\A\to \B$, such as
Quillen did in \cite{quillen} where he stated his celebrated
Theorems A and B for the classifying spaces of small categories.
Every object $b\in \Ob\B$ has an associated {\em homotopy fiber
bicategory} $F\!\!\downarrow{\!_b}$ whose objects are the 1-cells
$f:Fa\to b$ in $\B$, with $a$ an object of $\A$; the 1-cells consist
of all triangles
$$
\xymatrix@C=10pt@R=14pt{ Fa\ar[rr]^{Fu}\ar[rd]_{f}&
\ar@{}[d]|(.27){\beta}|(.48){\Rightarrow}&Fa'\ar[ld]^{f'}\\
 &b& }
$$
with $u:a\to a'$ a 1-cell in $\A$ and $\beta:f\Rightarrow f'\circ
Fu$ a 2-cell in $\B$, and the 2-cells of this bicategory are
commutative diagrams of 2-cells in $\B$ of the form
$$
\xymatrix@C=18pt@R=18pt{
&f\ar@2[ld]_{\beta}\ar@2[rd]^{\beta'}& \\
f'\circ Fu\ar@2[rr]^{1_{f'}\circ F\alpha}&&f'\circ Fu'
}
$$
with $\alpha:u\Rightarrow u'$ a 2-cell in $\A$. Compositions,
identities, and the structure associativity and unit constraints in
$F\!\!\downarrow{\!_b}$ are canonically provided by those of the
involved  bicategories and the structure 2-cells of the lax functor
(see Section \ref{theB} for details). For the case $F=1_\B$, we have
the {\em comma} bicategory $\B\!\!\downarrow{\!_b}$.  Then, we prove
(see  Theorem \ref{B}):
\begin{quote}{\em ``For every object $b$ of the bicategory $\B$, the induced square
$$
\begin{array}{c}
\xymatrix{\BB(F\!\!\downarrow{\!_b})\ar[d]\ar[r]&\BB(\B\!\!\downarrow{\!_b})
\ar[d]
\\
\BB\A\ar[r]^{\BB F}&\BB\B}\end{array}
$$
is homotopy cartesian if and only if all the  maps $\BB
p:\BB(F\!\!\downarrow{\!_b})\to \BB(F\!\!\downarrow{\!_{b'}})$,
induced by the 1-cells $p:b\to b'$ of $\B$, are homotopy
equivalences.\!"}
\end{quote}

Since the spaces $\BB(\B\!\!\downarrow{\!_b})$ are contractible
(Lemma \ref{cont}), the result above tells us that, under the
minimum necessary conditions, the classifying space of the homotopy
fiber bicategory $F\!\!\downarrow{\!_b}$ is homotopy equivalent to
the homotopy fiber of $\BB F:\BB\A\to\BB \B$ at its 0-cell $\BB b\in
\BB\B$. Thus, the name `homotopy fiber bicategory' is well chosen.
Furthermore, as a corollary, we obtain (see Theorem \ref{A}):
\begin{quote}{\em ``\! If all the spaces
$\BB(F\!\!\downarrow{\!_b})$ are contractible, then the map $\BB
F:\BB\A\to\BB \B$ is a homotopy equivalence.\!"}
\end{quote}

When the bicategories $\A$ and $\B$ involved in the results above
are actually categories, then they are reduced to the well-known
Theorems A and B by Quillen \cite{quillen}. Indeed, the methods used
in the proof of Theorem \ref{B} we give follow  similar lines to
those used by Quillen in his proof of Theorem B. However, the
situation with bicategories is more complicated than with
categories. Let us stress the two main differences between both
situations: On one hand, every 2-cell $\sigma:p\Rightarrow q:b\to
b'$ in $\B$ gives rise to a homotopy $$\BB\sigma:\BB p\simeq\BB
q:\BB(F\!\!\downarrow{\!_b})\to\BB(F\!\!\downarrow{\!_{b'}})$$ that
must be taken into account.  On the other hand,  for $p:b\to b'$ and
$p':b'\to b''$ any two composable 1-cells in $\B$,  we have a
homotopy
$$\BB p'\circ \BB p \simeq \BB(p'\circ p):\BB(F\!\!\downarrow{\!_b})\to \BB(F\!\!\downarrow{\!_{b''}}),$$
rather than the identity $\BB p'\circ \BB p = \BB(p'\circ p)$, as it
happens in the category case. This unfortunate behavior is due to
the fact that neither is the horizontal composition of 1-cells in
the bicategories involved  (strictly) associative nor does the lax
functor preserve (strictly) that composition. Therefore, in the
process of taking homotopy fiber bicategories, $F\!\!\downarrow:
b\mapsto F\!\!\downarrow{\!_b}$, we are forced to deal with {\em lax
bidiagrams of bicategories}
$$
\F:\B\to\Bicat,\hspace{0.3cm} b\mapsto \F_b,
$$
which are a type of lax functors in the sense of Gordon, Power and
Street \cite{g-p-s} from the bicategory $\B$ to the tricategory of
small bicategories, rather than ordinary diagrams of small
categories, that is, functors $\F:\B\to \Cat$, as it happens  when
both $\A$ and $\B$ are categories.

After this introductory Section 1, the paper is organized in four
sections. Section 2 is an attempt to make the paper as
self-contained as possible; hence, at the same time as we set
notations and terminology, we define and describe in detail the kind
of lax functors $\F:\B\to \Bicat$ we are going to work with. Section
3 is very technical but crucial to our discussions. It is mainly
dedicated to describing in detail a {\em bicategorical Grothendieck
construction}, which assembles any lax bidiagram of bicategories
$\F:\B \to\Bicat$ into a bicategory $\int_\B\F$. This is similar to
what the ordinary  construction, due to Grothendieck
\cite{groth,grothendieck}, Giraud \cite{giraud-2,giraud}, and
Thomason \cite{thomason} on lax diagrams  of categories with the
shape of any given category. By means of this higher Grothendieck
construction, in Section 4 we establish the third relevant result of
the paper, namely (see Theorem \ref{th1}):
\begin{quote}{\em ``\!If $\F:\B\to \Bicat$ is a lax bidiagram of bicategories such that each 1-cell
$p:b\to b'$ in the bicategory $\B$ induces a homotopy equivalence
$\BB\F_b\simeq \BB\F_{b'}$, then, for every object $b\in \Ob\B$,
there is an induced homotopy cartesian square
$$
\xymatrix{
\BB\F_b\ar[r]\ar[d]&\BB\!\int_\B\F\ar[d]\\ pt \ar[r]^{\BB b}&\BB\B.}
$$
That is, the classifying space $\BB\F_b$ is homotopy equivalent to
the homotopy fiber of the  map induced on classifying spaces by the
projection homomorphism $\int_\B\F \to \B$ at the 0-cell
corresponding to the object $b$.\!"}
\end{quote}

Thanks to Thomason's Homotopy Colimit Theorem \cite{thomason},  when
$\B$ is a small category and $\F$ values in $\Cat$, the result above
is equivalent to the relevant lemma used by Quillen in his proof of
Theorem B. Similarly here, the proof of the bicategorical Theorem B,
given in the last Section 5, essentially consists of two steps:
First, to apply that key result above to the lax bidiagram of
homotopy fiber bicategories, $F\!\!\downarrow\,:\B\to\Bicat$, of a
lax functor $F:\A\to \B$. Second,  to prove that there is a
 homomorphism $\int_\B F\!\!\downarrow\,\, \to \A$
inducing a homotopy equivalence $\BB(\int_\B F\!\!\downarrow)\simeq
\BB\A$, so that the bicategory $\int_\B F\!\!\downarrow$ may be
thought of as the ``total bicategory" of the lax functor $F$.
Section 5 also includes some applications to classifying spaces of
monoidal categories. For instance, we find a new proof of the
well-known result by Mac Lane \cite{mac} and Stasheff \cite{sta}:

\begin{quote}{\em ``\! Let
$(\mathcal{M},\otimes)=(\mathcal{M},\otimes,I,\aso,\bl,\br)$ be a
monoidal category. If multiplication for each object $x\in
\Ob\mathcal{M}$, $y\mapsto y\otimes x$, induces a homotopy
autoequivalence on $\BB\mathcal{M}$, then there is a homotopy
equivalence
$$\BB \mathcal{M}\simeq \Omega\BB(\mathcal{M},\otimes),$$ between the
classifying space of the underlying category and the loop space of
the classifying space of the monoidal category.\!"}
\end{quote}

\section{Bicategorical preliminaries: Lax bidiagrams of
bicategories}\label{laxbid} In this paper we shall work with small
bicategories, and we refer the reader to the papers by B\'{e}nabou
\cite{benabou},  Street \cite{street}, Gordon-Power-Street
\cite{g-p-s}, Gurski \cite{gurski}, and Leinster \cite{leinster},
for the background. The bicategorical conventions and the notations
that we use along the paper are the same as in \cite[\S 2.1]{ccg-2}
and \cite[\S2.4]{ccg-1}. Thus,  given  any bicategory $\B$, the
composition in each hom-category $\B(a,b)$, that is, the vertical
composition of 2-cells, is denoted by $\beta\cdot \alpha$, while the
symbol $\circ$ is used to denote the horizontal composition functors
$\B(b,c)\times \B(a,b) \overset{\circ}\to \B(a,c)$. Identities are
denoted as $1_f:f\Rightarrow f$, for any 1-cell $f:a\to b$, and
$1_a:a\to a$, for any object $a\in\Ob\B$. The associativity, right
unit, and left unit constraints of the bicategory are  respectively
denoted by the letters $\boldsymbol{a}$, $\boldsymbol{r}$, and
$\boldsymbol{l}$.

We will use that, in any bicategory, the commutativity of the two
triangles
\begin{equation}\label{tri2}\begin{array}{cc}\xymatrix@C=3pt@R=15pt{(1\circ g)\circ f
\ar@2{->}[rr]^{\aso}\ar@2{->}[rd]_{\bl\circ 1}&
&1\circ (g\circ f)
\ar@2{->}[ld]^{\bl}\\&g\circ f&}
&\xymatrix@C=5pt@R=20pt{(g\circ f)\circ 1
\ar@2{->}[rr]^{\aso}\ar@2{->}[rd]_{\br}&
&g\circ(f\circ 1)
\ar@2{->}[ld]^{1\circ \,\br}\\&g\circ f&}
\end{array}
\end{equation}
and the equality
\begin{equation}\label{equide}
\br_1=\bl_1: 1\circ 1\cong 1
\end{equation}
are consequence of the other axioms (this is not obvious, but a
proof can be done paralleling the given, for monoidal categories, by
Joyal and Street in \cite[Proposition 1.1]{joyal-street}).

A {\em lax functor} is usually written as a pair $
F=(F,\widehat{F}):\A \to \B$ since we will generally denote its
structure constraints by $$\widehat{F}_{g,f}:Fg\circ Ff\Rightarrow
F(g\circ f),\hspace{0.4cm} \widehat{F}_a:1_{Fa}\Rightarrow F1_a.$$
The lax functor is termed a {\em pseudo functor} or {\em
homomorphism} whenever all the structure constraints $\widehat{F}$
are invertible. If the unit constraints $\widehat{F}_a$ are all
identities, then the lax functor is qualified as (strictly)  {\em
unitary} or {\em normal} and if, moreover, the constraints
$\widehat{F}_{g,f}$ are also identities, then $F$ is called a
$2$-{\em functor}.

If $F,G:\A \to \B$ are lax functors, then we follow the convention
of \cite{g-p-s} in what is meant by a {\em lax transformation}
${\alpha=(\alpha,\widehat{\alpha}):F\Rightarrow G}$. Thus, $\alpha$
consists of morphisms ${\alpha a:Fa\to Ga}$, $a\in \Ob\A$, and of
2-cells $\widehat{\alpha}_f:\alpha b\circ Ff\Rightarrow Gf\circ
\alpha a$, subject to the usual axioms. When the 2-cells
$\widehat{\alpha}$ are all invertible, we say that
$\alpha:F\Rightarrow G$ is a {\em pseudo transformation}.

In accordance with the orientation of the naturality 2-cells chosen,
if ${\alpha,\beta:F\Rightarrow G}$ are two lax transformations, then
a {\em modification} $\sigma:\alpha\Rrightarrow\beta$ will consist
of 2-cells $\sigma a:\alpha a\Rightarrow\beta a$, $a\in
\mbox{Ob}\A$, subject to the commutativity condition, for any
morphism $f:a\to b$ of $\A$:
 $$
\begin{array}{c}
\xymatrix@R=35pt@C=35pt{
Fa
\ar@/_1.2pc/[r]
\ar@{}@<-19pt>[r]|(.4){\alpha a}
\ar@{}[r]_{\Uparrow\sigma}
\ar[r]^{\beta a}\ar[d]_{F\!f}
&Ga\ar[d]^{G\!f}
\ar@{}@<30pt>[d]|{ =}
\\ Fb
\ar@{}@<-25pt>[u]_(.26){\Rightarrow}_(.4){\widehat{\alpha}}
\ar[r]_{\alpha b}&Gb
}\hspace{0.3cm}
\xymatrix@R=35pt@C=35pt{
Fa\ar[r]^{\beta a}\ar[d]_{F\!f}
\ar@{}@<35pt>[d]_(.4){\Rightarrow}_(.25){\widehat{\beta}}
&Ga\ar[d]^{G\!f}
\\ Fb\ar@/^1.2pc/[r]
\ar@{}@<19pt>[r]|(.6){\beta b}
\ar@{}[r]^{\Uparrow\sigma}
\ar[r]_{\alpha b}&Gb.
}
\end{array}
$$

$\Bicat$ denotes the tricategory of of bicategories,  homomorphisms,
pseudo natural transformations, and modifications. In the structure
of $\Bicat$ we use, the composition of pseudo transformations is
taken to be
$$\big(\xymatrix@C=8pt{\B \ar@/^0.6pc/[rr]^{G}\ar@{}[rr]|{\Downarrow \beta}
\ar@/_0.6pc/[rr]_{ G'} &  &\C}\big)\big(\xymatrix@C=8pt{\A
\ar@/^0.6pc/[rr]^{F}\ar@{}[rr]|{\Downarrow \alpha}
\ar@/_0.6pc/[rr]_{ F'} &  &\B}\big)=\big(\xymatrix@C=9pt{\A
\ar@/^0.7pc/[rr]^{GF}\ar@{}[rr]|{\Downarrow \beta\alpha}
\ar@/_0.7pc/[rr]_{ G'\!F'} &  &\C}\big), $$
 where  $\beta
\alpha=\beta F'\circ G\alpha:\big(\xymatrix{GF\ar@2[r]^{G\alpha}&
GF'\ar@2[r]^{\beta F'}& G'\!F'}\big)$, but note the existence of the
useful invertible modification
\begin{equation}\label{4}\begin{array}{l}\xymatrix@R=15pt@C=20pt{GF\ar@{}@<-4pt>[rd]^(.5){\Rrightarrow}
\ar@{=>}[r]^{\beta F}\ar@{=>}[d]_{G\alpha}&G'F
\ar@{=>}[d]^{G'\alpha}\\ GF'\ar@{=>}[r]^{\beta
F'}&G'F'}\end{array}\end{equation} whose component  at an object $a$
of $\A$, is $\widehat{\beta}_{\alpha a}$, the component of $\beta$
at the morphism $\alpha a$.

\subsection{Lax bidiagrams of bicategories}

The next concept of fibered bicategory in bicategories is the basis
of most of our subsequent discussions.
 Let $\B$ be a bicategory. Regarding $\B$ as a tricategory in which
the $3$-cells are all identities, we define a \emph{lax bidiagram of
bicategories}
 \begin{equation}\label{lb}\F=(\F, \chi,\xi,\omega,\gamma,\delta):\B^{\mathrm{op}}\to\Bicat\end{equation}
 to be a contravariant lax functor of tricategories from $\B$ to  $\Bicat$,
 all of whose coherence modifications are invertible. More explicitly, a lax bidiagram of bicategories $\F$ as
above consists of the following data:

\vspace{0.2cm} $(\mathbf{D1})$ for each object $b$ in $\mathcal{B}$,
a bicategory \hspace{0.1cm}$\F_b$;

\vspace{0.2cm} $(\mathbf{D2})$ for each 1-cell $f:a\to b$ of $\B$, a
homomorphism \hspace{0.1cm}$f^*:\F_b\to \F_a$;

\vspace{0.2cm} $(\mathbf{D3})$ for each 2-cell $\xymatrix@C=0.5pc{a
\ar@/^0.7pc/[rr]^{ f} \ar@/_0.7pc/[rr]_{
g}\ar@{}[rr]|{\Downarrow\alpha} &
 &b }$ of $\B$, a pseudo transformation \hspace{0.1cm}$\alpha^*:f^*\Rightarrow
 g^*$;
 \vspace{-0.5cm}$$\xymatrix@C=0.6pc{\F_b  \ar@/^0.8pc/[rr]^{ f^*}
\ar@/_0.8pc/[rr]_{ g^*}\ar@{}[rr]|(.55){\Downarrow\alpha^*} &
&\F_a}$$

$(\mathbf{D4})$ for each  two composable 1-cells
$\xymatrix@C=13pt{a\ar[r]^{f}&b\ar[r]^{g}&c}$ in the bicategory
$\B$, a pseudo transformation
\hspace{0.1cm}$\chi_{_{g,f}}:f^*g^*\Rightarrow (g\circ f)^*$;
$$\xymatrix@C=10pt@R=15pt{
&\F_c\ar[ld]_{g^*}\ar[rd]^{(g\,\circ f)^*}\ar@{}[d]|(.65){\Rightarrow}
\ar@{}[d]|(.4){\chi}& \\
\F_b\ar[rr]_{f^*}&&\F_a}$$

$(\mathbf{D5})$ for each object $b$ of $\B$, a pseudo transformation
\hspace{0.1cm}$\chi_{_b}: 1_{\F_b}\Rightarrow 1_b^*$;
$$\xymatrix@C=0.6pc{\F_b  \ar@/^0.8pc/[rr]^{ 1_{\F_b}}
\ar@/_0.8pc/[rr]_{ 1_b^*}\ar@{}[rr]|{\Downarrow\chi} &  &\F_b}$$

\vspace{-0.6cm}$(\mathbf{D6})$ for any two vertically composable 2-cells
$\xymatrix{a\ar@/^/[r]^f \ar@/_/[r]_g \ar@{}[r]|{\Downarrow\alpha} & b}$ and $\xymatrix{a\ar@/^/[r]^g \ar@/_/[r]_h \ar@{}[r]|{\Downarrow\beta} & b}$
in $\B$,
 an invertible modification $\xi_{_{\beta,\alpha}}\!:\beta^*\circ \alpha^*
 \Rrightarrow  (\beta\cdot  \alpha)^*$;
$$\xymatrix@C=15pt@R=15pt{
&{f}^*\ar@2[ld]_{\alpha^*}\ar@2[rd]^{(\beta\cdot\alpha)^*}\ar@{}[d]|(.65){\Rrightarrow}
\ar@{}[d]|(.4){\xi}& \\
g^*\ar@2[rr]_{\beta^*}&&{h}^*}$$

$(\mathbf{D7})$ for each 1-cell $f:a\to b$ of $\B$, an invertible
modification \hspace{0.1cm}$\xi_{_f}:1_{f^*}\Rrightarrow 1_f^*$;
$$
\xymatrix{f^* \ar@{=>}@/_1.2pc/[d]_{1_{f^*} }
\ar@{=>}@/^1.2pc/[d]^{1_{f}^* }\ar@{}[d]|{\Rrightarrow}\ar@{}[d]|(.32){\xi}\\f^*}
$$

$(\mathbf{D8})$ for every two horizontally composable 2-cells
$\xymatrix@C=0.5pc{a\ar@/^0.7pc/[rr]^f \ar@/_0.7pc/[rr]_{h}
\ar@{}[rr]|{\Downarrow\alpha}&  &
b\ar@/^0.7pc/[rr]^{g}\ar@/_0.7pc/[rr]_{k}
\ar@{}[rr]|{\Downarrow\beta}& & c}$, an invertible modification
\hspace{0.1cm}$\chi_{_{\beta,\alpha}}:(\beta\circ\alpha)^*\!\circ
\chi_{_{g,f}}\Rrightarrow\chi_{_{k,h}}\circ
 (\alpha^*\beta^*)$;
$$
\xymatrix{f^*\,g^*\ar@2[r]^{\alpha^*\beta^*}\ar@2[d]_\chi\ar@{}@<-22pt>[r]|{\Rrightarrow}
\ar@{}@<-15pt>[r]|{\chi}
           & h^*\,k^*\ar@2[d]^{\chi}\\
                    (g\circ f)^*\ar@2[r]_{(\beta\circ \alpha)^*} & (k\circ h)^*}
$$

$(\mathbf{D9})$ for every three composable 1-cells
$\xymatrix@C=13pt{a\ar[r]^{f}&b\ar[r]^g&c\ar[r]^h&d}$ in $\B$, an
invertible modification \hspace{0.1cm}$\omega_{_{h,g,f}}:\,
\aso^*\!\circ (\chi_{_{h\circ g,f}}\circ f^*\chi_{_{h,g}}
)\Rrightarrow \chi_{_{h,g\circ f}}\circ \chi_{_{g,f}}h^*$;
$$
\xymatrix@C=20pt{f^*g^*h^*\ar@2[rr]^{\chi h^*}\ar@2[d]_{f^*\chi}
\ar@{}@<-22pt>[rr]|{\Rrightarrow}\ar@{}@<-16pt>[rr]|{\omega}&&(g\circ f)^*h^*\ar@2[d]^{\chi}\\
f^*(h\circ g)^*\ar@2[r]^-{\chi}&((h\circ g)\circ f)^* \ar@2[r]^{\aso^*}&(h\circ (g\circ f))^* }
$$

$(\mathbf{D10})$ for any 1-cell $f:a\to b$ of $\B$, two invertible
modifications
$${\gamma_f:\bl^*_f\circ (\chi_{_{1_b,f}}\circ
f^*\chi_{_b})\Rrightarrow 1_{f^*}},\ \  \delta_f:
\br^*_f\circ (\chi_{_{f,1_a}}\circ \chi_{_a}f^*)\Rrightarrow 1_{f^*}.$$
$$
\xymatrix{ f^*1_b^*\ar@2[d]_{\chi}\ar@{}@<-22pt>[r]|{\Rrightarrow}
\ar@{}@<-15pt>[r]|{\gamma}& f^*\ar@2[d]^{1_{f^*}}\ar@2[l]_-{f^*\!\chi}
\ar@2[r]^-{\chi f^*} \ar@{}@<-22pt>[r]|{\Lleftarrow}
\ar@{}@<-15pt>[r]|{\delta}
&1_a^*f^*
\ar@2[d]^{\chi}\\ (1_b\circ f)^*\ar@2[r]_{\bl^*}&f^*&(f\circ 1_a)^*\ar@2[l]^{\br^*}    }
$$

\vspace{0.2cm} These data must satisfy the following coherence
conditions:

\vspace{0.2cm} $(\mathbf{C1})$ for any three composable $2$-cells
$\xymatrix@C=13pt{f\ar@2[r]^{\alpha}&g\ar@2[r]^{\beta}&h\ar@2[r]^(.3){\zeta}&k:a\to
b}$ in $\B$, the  equation on modifications below holds;
$$
\xymatrix@C=35pt@R=35pt{g^*\ar@2[d]_{\beta^*}
\ar@{}@<-15pt>[r]|(0.35){\Rrightarrow}\ar@{}@<-9pt>[r]|(0.3){\xi}&
f^*\ar@2[l]_{\alpha^*}\ar@2[d]^{(\zeta\cdot\beta\cdot\alpha)^*}\ar@2[ld]|{_{(\beta\cdot\alpha)^*}}&
\ar@{}@<5pt>[d]|{\textstyle =} &
g^*\ar@2[d]_{\beta^*}\ar@2[rd]|{_{(\zeta\cdot\beta)^*}}
\ar@{}@<-15pt>[r]|(0.7){\Rrightarrow}
\ar@{}@<-9pt>[r]|(0.65){\xi}&k^*\ar@2[l]_{\alpha^*}\ar@2[d]^{(\zeta\cdot\beta\cdot\alpha)^*}\\
h^*\ar@2[r]_{\zeta^*}\ar@{}@<10pt>[r]|(0.75){\Rrightarrow}\ar@{}@<17pt>[r]|(0.7){\xi}&
k^*& & h^*\ar@{}@<10pt>[r]|(0.3){\Rrightarrow}\ar@{}@<17pt>[r]|(0.25){\xi}\ar@2[r]_{\zeta^*}&k^*}
$$

$(\mathbf{C2})$ for any 2-cell
$\xymatrix@C=13pt{f\ar@2[r]^(.26){\alpha}&g:a\to b}$ of $\B$,
$$
\xymatrix{\ar@{}@<-9pt>[rr]^(.2){1_{f^*}}\ar@{}@<20pt>[d]|(.53){\Rrightarrow}|(.36){\xi}&
f^*\ar@2@/_0.9pc/[ld]\ar@2[rd]^{\alpha^*}
\ar@{=>}@/^0.7pc/[ld]
\ar@{}@<8pt>[d]|(.7){\Rrightarrow}|(.5){\xi}
& \ar@{}@<10pt>[d]|(.5){\textstyle =}\ar@{}@<28pt>[d]|(.5){\textstyle \br_{\alpha^*}\,,}
\\f^*\ar@2[rr]_{\alpha^*}\ar@{}@<0.5pt>[rr]^(.37){1_f^*}&&g^*} \hspace{0.4cm}
\xymatrix{\ar@{}@<-9pt>[rr]^(.2){1_{g^*}}\ar@{}@<20pt>[d]|(.53){\Rrightarrow}|(.36){\xi}&
g^*\ar@{}@<8pt>[d]|(.7){\Rrightarrow}|(.5){\xi} &
\ar@{}@<10pt>[d]|(.5){\textstyle =}\ar@{}@<28pt>[d]|(.5){\textstyle
\bl_{\alpha^*};}
\\g^*\ar@2@/_0.7pc/[ru] \ar@{=>}@/^0.9pc/[ru]
\ar@{}@<0.7pt>[rr]^(.42){1_{\!{g}}^*}&&{f}^*\ar@2[lu]_{\alpha^*}
\ar@2[ll]^{\alpha^*}}
$$

\begin{notation}{\em Thanks to conditions $(\mathbf{C1})$ and $(\mathbf{C2})$, for
each objects $a,b\in \Ob\B$, we have a homomorphism $\B(a,b)\to
\Bicat(\F_b,\F_a)$ such that

\vspace{-0.3cm}$$\xymatrix@C=0.5pc{a \ar@/^0.7pc/[rr]^{ f}
\ar@/_0.7pc/[rr]_{g}\ar@{}[rr]|{\Downarrow\alpha} &
 &b } \mapsto \xymatrix@C=0.6pc{\F_b  \ar@/^0.8pc/[rr]^{ f^*}
\ar@/_0.8pc/[rr]_{ g^*}\ar@{}[rr]|(.55){\Downarrow\,\alpha^*} &
&\F_a,}$$ and whose structure constraints are the deformations $\xi$
in $(\mathbf{D6})$ and $(\mathbf{D7})$. Then, whenever it is given a
commutative diagram in the category $\B(a,b)$ of the form
\begin{equation}\label{dn1}\xymatrix@C=15pt@R=20pt{f\ar@2[d]_{
\alpha_0}\ar@2[r]^{\beta_0}&g_1\ar@2[r]^{\beta_1}&\cdots\ar@2[r]&
g_n\ar@2[d]^{\beta_n}\\
f_1\ar@2[r]^{\alpha_1}&\cdots\ar@2[r]&f_m\ar@2[r]^{
\alpha_m}&g,}\end{equation} we will denote by
\begin{equation}\label{dn2}\xymatrix@C=15pt@R=20pt{f^*\ar@2[d]_{\alpha^*_0}
\ar@{}@<-20pt>[rrr]|{\cong} \ar@{}@<-13pt>[rrr]|{\xi}
\ar@2[r]^{\beta_0^*}&{g}^*_1\ar@2[r]^{{\beta}^*_1} &\cdots\ar@2[r]&
{g}^*_n\ar@2[d]^{{\beta}^*_n}\\
f_1^*\ar@2[r]^{\alpha_1^*}&\cdots\ar@2[r]&f_m^*\ar@2[r]^{
\alpha_m^*}&g^*}\end{equation} the invertible modification obtained
by an (any) appropriate composition of the modifications $\xi$ and
their inverses $\xi^{-1}$, once any particular bracketing in the
strings $\alpha_0^*,\ldots,\alpha_m^*$ and
$\beta_0^*,\ldots,\beta_n^*$ has been chosen. That diagram
(\ref{dn2}) is well defined from diagram (\ref{dn1}) is a
consequence of the coherence theorem for homomorphisms of
bicategories \cite[Theorem 1.6]{g-p-s}.

Furthermore, for any diagram  $\xymatrix@C=0.8pc{a\ar[r]^{f}&b
\ar@/^0.7pc/[rr]^{g}
\ar@/_0.7pc/[rr]_{g'}\ar@{}[rr]|{\Downarrow\alpha} &
 &c\ar[r]^h&d }$ in $\B$, we shall denote by
$$\chi_{_{\alpha,f}}: (\alpha\circ 1_f)^*\circ \chi_{_{g,f}}\Rrightarrow \chi_{_{g',f}}\circ f^*\alpha^*,\hspace{0.4cm}
 \chi_{_{h,\alpha}}: (1_h\circ \alpha)^*\circ \chi_{_{h,g}}\Rrightarrow \chi_{_{h,g'}}\circ \alpha^*h^* ,
$$
  $$
  \xymatrix@C=30pt{f^*g^*\ar@{}@<34pt>[d]|(.55){\Rrightarrow}
  \ar@{}@<34pt>[d]|(.4){\chi}
  \ar@2[r]^{f^*\!\alpha^*}\ar@2[d]_{\chi}&f^*g'^*\ar@2[d]^{\chi}\\
  (g\circ f)^*\ar@2[r]_{(\alpha\circ 1_f)^*}&(g'\circ f)^*
  } \hspace{0.4cm}
  \xymatrix@C=30pt{g^*h^*\ar@{}@<34pt>[d]|(.55){\Rrightarrow}
  \ar@{}@<34pt>[d]|(.4){\chi}
  \ar@2[r]^{\alpha^*\!h^*}\ar@2[d]_{\chi}&g'^*h^*\ar@2[d]^{\chi}\\
  (h\circ g)^*\ar@2[r]_{(1_h\circ \alpha)^*}&(h\circ g')^*
  }
  $$
the modifications obtained, respectively, by pasting the diagrams in
$\Bicat$ below.
$$
 \xymatrix@C=45pt{f^*g^*
 \ar@{}@<38pt>[dd]|(.75){\Rrightarrow}
  \ar@{}@<38pt>[dd]|(.65){\chi}
  \ar@2[dd]_{\chi}
 \ar@2[r]|(.6){\ 1_{\!f^*}\!\alpha^*}\ar@{}@<3pt>[r]^(.45){\cong}
 \ar@{}@<-16pt>[r]|{\cong}\ar@{}@<-9pt>[r]|{\xi 1_{\alpha^{\!*}}}
 \ar@2@/^1.5pc/[r]^{f^*\!\alpha^*}\ar@2@/_2pc/[r]_{1_f^*\alpha^*}&f^*g'^*
 \ar@2[dd]^{\chi} \\
 &\\
(g\circ f)^*\ar@2[r]_{(\alpha\circ 1_f)^*}&(g'\circ f)^*&
 }\hspace{0.4cm}
  \xymatrix@C=45pt{g^*h^*
 \ar@{}@<38pt>[dd]|(.75){\Rrightarrow}
  \ar@{}@<38pt>[dd]|(.65){\chi}
  \ar@2[dd]_{\chi}
 \ar@2[r]|(.6){\ \!\alpha^*\!1_{\!h^*}}
 \ar@{}@<3pt>[r]^(.45){\cong}
 \ar@{}@<-17pt>[r]|{\cong}\ar@{}@<-9pt>[r]|{1_{\alpha^{\!*}}\xi }
 \ar@2@/^1.5pc/[r]^{\alpha^*\!h^*}\ar@2@/_2pc/[r]_{\alpha^*1_h^*}&g'^*h^*
 \ar@2[dd]^{\chi} \\
 &\\
(h\circ g)^*\ar@2[r]_{(1_h\circ \alpha)^*}&(h\circ g')^*&
 }
$$
 }
\end{notation}

$(\mathbf{C3})$ for every diagram of $2$-cells   $
\xymatrix@C=2.5pc{a\ar@/^1.2pc/[r]^(.4){f}_{}="a"\ar[r]|(.38){\,f'}^{}="b"_{}="e"
\ar@/_1.2pc/[r]_(.4){f''}^{}="c" & b
\ar@/^1.3pc/[r]^(.4){g}_{}="d"\ar[r]|(.38){\,g'}^{}="f"_{}="g"\ar@/_1.2pc/[r]_(.4){g''}^{}="h"
                     & c
                    \ar@2"a";"b"^\alpha
                    \ar@2"e";"c"^{\alpha'}
                    \ar@2"d";"f"^{\beta}
                    \ar@2"g";"h"^{\beta'}}
$ in $\mathcal{B}$,
$$
\xymatrix@C=15pt@R=30pt{
(g\circ f)^*
\ar@{}@<-20pt>[rr]|(.36){\Rrightarrow}
\ar@{}@<-13pt>[rr]|(.36){\chi}
\ar@2[d]_{(\beta\circ\alpha)^*}&&&f^*g^*
\ar@{}@<45pt>[dd]|{\textstyle =}
\ar@2[lll]_{\chi} \ar@2@/^1.5pc/[dd]|(.3){\ \
(\alpha'\cdot\alpha)^*(\beta'\cdot\beta)^*}
\ar@2[lld]_{\alpha^*\beta^*}
\ar@2@/_1.5pc/[dd]|(.6){(\!\alpha'^*\!\circ
\alpha^*\!)(\!\beta'^*\!\circ\beta^*\!)}
\\
 (g'\circ f')^*\ar@2[d]_{(\beta'\circ\alpha')^*}
\ar@{}@<-20pt>[rr]|(.36){\Rrightarrow}
\ar@{}@<-13pt>[rr]|(.36){\chi}
 &f'^*g'^*\ar@2[rrd]_{\alpha'^*\beta'^*}\ar@2[l]_-{\chi}
 \ar@{}@<4pt>[r]|(1.0){\overset{(\ref{4})}\cong}
 \ar@{}[rr]|(1.0){\cong}
 \ar@{}@<7pt>[rr]|(1.0){\xi\xi}&&\\
 (g''\circ f'')^*&&&f''^*g''^*\ar@2[lll]^{\chi}}
 \xymatrix@C=15pt@R=30pt{
(g\circ f)^* \ar@2@/^2.3pc/[dd]^(.3){((\beta'\cdot\beta)\circ (\alpha'\cdot\alpha))^*}
\ar@2[d]_{(\beta\circ\alpha)^*}&&&f^*g^*\ar@2[lll]_{\chi}
\ar@2@/^1.5pc/[dd]|(.3){\ \ (\alpha'\cdot\alpha)^*(\beta'\cdot\beta)^*}
 \\
 (g'\circ f')^*\ar@2[d]_{(\beta'\circ\alpha')^*}
 \ar@{}@<10pt>[r]|(.4){\cong}
 \ar@{}@<10pt>[r]^(.4){\xi}
 \ar@{}@<-5pt>[rr]|(.9){\Rrightarrow}
 \ar@{}@<-4pt>[rr]^(.9){\chi}
 &&\\
 (g''\circ f'')^*&&&f''^*g''^*\ar@2[lll]^{\chi}}
 $$

$(\mathbf{C4})$ for every pair of composable 1-cells
$\xymatrix@C=13pt{a\ar[r]^{f}&b\ar[r]^g&c}$,

$$
\xymatrix{(g\!\circ\!f)^*
\ar@{=>}@/_1.2pc/[d]
\ar@{}@<-27pt>[d]|{1_{(g\circ f)^*}}
\ar@{=>}@/^1.2pc/[d]^{1_{g\circ f}^*}
\ar@{}[d]|{\Rrightarrow}\ar@{}[d]|(.32){\xi}& f^*g^*\ar@2[l]_{\chi}
\ar@{}@<-10pt>[d]|{\Rrightarrow}\ar@{}@<-10pt>[d]|(.32){\chi}
\ar@{=>}@/^1.2pc/[d]^{1_{f}^*1_{g}^*}
\ar@{}@<50pt>[d]|{\textstyle =}
\\(g\!\circ\!f)^*
&
 f^*g^*\ar@2[l]_{\chi}
}
\hspace{0.4cm}
\xymatrix{(g\!\circ\!f)^* \ar@{=>}@/_1.2pc/[d]
\ar@{}@<-27pt>[d]|{1_{(g\circ f)^*}}
\ar@{}@<8pt>[d]|(.45){\cong}& f^*g^*
\ar@{}[d]|{\Rrightarrow}\ar@{}[d]|(.32){\xi\xi}
\ar@2[l]_{\chi} \ar@{=>}@/^1.2pc/[d]^{1_{f}^*1_{g}^*}
\ar@{=>}@/_1.2pc/[d] \ar@{}@<-26pt>[d]|(.6){1_{f^*}\!1_{g^*}}
\\(g\!\circ\!f)^*
&
 f^*g^*\ar@2[l]_{\chi}
}
$$

$(\mathbf{C5})$ for every $2$-cells
$\xymatrix@C=0.5pc{a\ar@/^0.7pc/[rr]^f \ar@/_0.7pc/[rr]_{f'}
\ar@{}[rr]|{\Downarrow\alpha}&  &
b\ar@/^0.7pc/[rr]^{g}\ar@/_0.7pc/[rr]_{g'}
\ar@{}[rr]|{\Downarrow\beta}& & c\ar@/^0.7pc/[rr]^h
\ar@/_0.7pc/[rr]_{h'} \ar@{}[rr]|{\Downarrow\zeta}&  & d}$, the
equation $A=A'$ holds, where
$$
\xymatrix@C=13pt@R=35pt{\ar@{}[dd]|{\textstyle A=}& &((h\circ g)\circ f)^*
\ar@2[d]|{((\zeta\circ\beta)\circ\alpha)^*}
\ar@{}@<-22pt>[r]|{\Rrightarrow}\ar@{}@<-16pt>[r]|{\chi}
\ar@2[ld]_{\aso^*}&f^*(h\circ g)^* \ar@{}@<-20pt>[r]|(.75){\Rrightarrow}
\ar@{}@<-13pt>[r]|(.75){\alpha^*\!\chi}
\ar@2[l]_(.45){\chi} \ar@2[d]|{\alpha^*(\zeta\circ \beta)^*}&
&f^*g^*h^*\ar@2[ll]_{f^*\chi}\ar@2[d]^{(\alpha^*\beta^*)\zeta^*}
\ar@{}@<-6pt>[d]|(.5){\cong}\ar@2@/_1.2pc/[d]\ar@{}@<-6pt>[d]_(.64){\alpha^*(\beta^*\zeta^*)\ \ }\\ &
(h\circ(g\circ f))^*\ar@{}[r]|{\cong}\ar@{}@<7pt>[r]|{\xi}
\ar@2[rd]\ar@{}@<-4pt>[rd]_{(\zeta\circ (\beta\circ \alpha))^*} &
((h'\circ g')\circ f')^*
\ar@{}@<-22pt>[rr]|(.4){\Rrightarrow}\ar@{}@<-16pt>[rr]|(.4){\omega}
\ar@2[d]^{\aso^*}&f'^*(h'\circ g')^*\ar@2[l]^(.4){\chi}& &
f'^*g'^*h'^*\ar@2[lld]^{\chi h'^* }\ar@2[ll]^{f'^*\chi}\\ &
& (h'\circ (g'\circ f'))^*& (g'\circ f')^*h'^*\ar@2[l]_-{\chi}& &
}
$$
$$
\xymatrix@C=13pt@R=35pt{\ar@{}[dd]|{\textstyle A'=} & &((h\circ g)\circ f)^* \ar@2[ld]_{\aso^*}
\ar@{}@<-22pt>[rr]|(.3){\Rrightarrow}\ar@{}@<-16pt>[rr]|(.3){\omega}&f^*(h\circ g)^*\ar@2[l]_-{\chi}
& &
f^*g^*h^*\ar@2[ll]_{f^*\chi}\ar@2[d]^{(\alpha^*\beta^*)\zeta^*}
\ar@2[lld]_{\chi h^*}\\ &
(h\circ(g\circ f))^*
\ar@{}@<-22pt>[rr]|(.6){\Rrightarrow}\ar@{}@<-16pt>[rr]|(.6){\chi}
\ar@2[rd]\ar@{}@<-4pt>[rd]_{(\zeta\circ (\beta\circ \alpha))^*}
&&(g\circ f)^*h^* \ar@2[d]^(.35){(\beta\circ \alpha)^*\zeta^*}
\ar@2[ll]_{\chi}
\ar@{}@<-4pt>[rr]|(.5){\Rrightarrow}\ar@{}@<3pt>[rr]|(.5){\chi\zeta^*}&&
f'^*g'^*h'^*\ar@2[lld]^{\chi h'^*}\\ &
& (h'\circ (g'\circ f'))^*& (g'\circ f')^*h'^*\ar@2[l]_-{\chi}&&
}
$$

$(\mathbf{C6})$ for every four composable 1-cells
$\xymatrix@C=13pt{a\ar[r]^{f}&b\ar[r]^g&c\ar[r]^h&d\ar[r]^k&e}$, the
equation $B=B'$ holds, where

$$
\xymatrix@C=5pt{B=&&f^*g^*h^*k^*\ar@2[rrd]\ar@{}@<3pt>[rrd]^{\chi h^* k^*}
                        \ar@2[lld] \ar@{}@<-3pt>[lld]_{f^*g^*\chi}
                                                                       &&\\
                    f^*g^*(k\!\circ\!h)^*\ar@2[d]_{f^*\chi}
                     \ar@{}@<18pt>[rrrr]|{\overset{(\ref{4})}\cong}
                    \ar@2[rr]^{\chi(k\circ h)^*}
                                                     &&
                                                     (g\!\circ\!f)^*(k\!\circ\!h)^*\ar@2[dd]^{\chi}
                        &&(g\!\circ\!f)^*h^*k^*\ar@2[d]^{\chi k^*}
                         \ar@2[ll]_{(g\circ f)^*\chi}\\
                    f^*((k\!\circ\!h)\!\circ\!g)^* \ar@{}@<-2pt>[rr]|{\Rrightarrow}\ar@{}@<4pt>[rr]|{\omega}
                    \ar@2[d]_{\chi}
                    &
                    &\ar@{}@<-2pt>[rr]|{\Rrightarrow}\ar@{}@<4pt>[rr]|{\omega}&
                        &(h\!\circ\!(g\!\circ\!f))^*k^*\ar@2[d]^{\chi}
                                     \\
                    (((k\!\circ\!h)\!\circ\!g)\!\circ\!f)^*\ar@2[rr]^{\aso^*}
                    \ar@2[rd]\ar@{}@<-3pt>[rd]_{(\aso\circ 1_f)^*}
                      &
                        &((k\!\circ\!h)\!\circ\!(g\!\circ\!f))^*\ar@2[rr]^{\aso^*}\ar@{}[d]|(.55){\cong}|(.4){\xi}&
                        &(k\!\circ\!(h\!\circ\!(g\!\circ\!f)))^* \\
 &((k\!\circ\!(h\!\circ\!g))\!\circ\!f)^*\ar@2[rr]_{\aso^*}&&(k\!\circ\!((h\!\circ\!g)\!\circ\!f))^*
 \ar@2[ru]_{\ (1_k\circ \aso)^*}& }
$$
$$
\xymatrix@C=10pt{B'=&&f^*g^*h^*k^*\ar@2[rrd]\ar@{}@<3pt>[rrd]^{\chi h^* k^*}
                        \ar@2[lld] \ar@{}@<-3pt>[lld]_{f^*g^*\chi}
                                                \ar@2[d]^{f^*\!\chi \,k^*}
                        &&\\
                    f^*g^*(k\!\circ\!h)^*\ar@2[d]_{f^*\chi}
                                \ar@{}[rr]|{\Rrightarrow} \ar@{}@<7pt>[rr]|{f^*\!\omega}
                      &&f^*(h\!\circ\!g)^*k^*\ar@2[rd]^{\chi k^*}
                                     \ar@2[ld]_{f^*\chi }
                                                 \ar@{}[rr]|{\Rrightarrow}\ar@{}@<7pt>[rr]|{\omega k^*}
                        &&(g\!\circ\!f)^*h^*k^*\ar@2[d]^{\chi k^*}\\
                    f^*((k\!\circ\!h)\!\circ\!g)^* \ar@2[r]^{f^*\!\aso^*}
                                        \ar@2[d]_{\chi}
                                \ar@{}@<44pt>[dd]|(.45){\Rrightarrow}
                                \ar@{}@<44pt>[dd]|(.36){\chi}
                     &f^*(k\!\circ\!(h\!\circ\!g))^*
                                           \ar@2[dd]^{\chi}
                                \ar@{}@<-20pt>[rr]_{\Rrightarrow}
                                \ar@{}@<-15pt>[rr]_{\omega}
                        &&((h\!\circ\!g)\!\circ\!f)^*k^*
                         \ar@2[r]^{\aso^*k^*}
                        \ar@2[dd]_{\chi}
                            \ar@{}@<30pt>[dd]|(.5){\Rrightarrow}
                        \ar@{}@<30pt>[dd]|(.42){\chi}
                        &(h\!\circ\!(g\!\circ\!f))^*k^*\ar@2[d]^{\chi}
                                     \\
                    (((k\!\circ\!h)\!\circ\!g)\!\circ\!f)^*
                    \ar@2[rd]\ar@{}@<-3pt>[rd]_{(\aso\circ 1_f)^*}
                      &
                        &&
                        &(k\!\circ\!(h\!\circ\!(g\!\circ\!f)))^* \\
                        &((k\!\circ\!(h\!\circ\!g))\!\circ\!f)^*\ar@2[rr]_{\aso^*}&&
                        (k\!\circ\!((h\!\circ\!g)\!\circ\!f))^*\ar@2[ru]_{\ (1_k\circ \aso)^*}& }
$$

$(\mathbf{C7})$ for every  2-cell
$\xymatrix@C=13pt{f\ar@2[r]^(.26){\alpha}&g:a\to b}$, the following
two equations on modifications hold:
$$
\xymatrix{&1_a^*f^*\ar@2[ld]_{\chi}
\ar@{}@<-3pt>[d]|(.55){\Rrightarrow}|(.38){\delta}
& f^*\ar@2[l]_(.4){\chi\!f^*}\ar@2[ld]\ar@{}@<3pt>[ld]_(.4){1_{f^*}}
\ar@2[d]^(.4){\alpha^*}\\
(f\!\circ\!1_a)^*
\ar@{}@<30pt>[d]|(.56){\cong}|(.38){\xi}
\ar@2[r]^{\br^*}\ar@2[d]_{(\alpha\circ
1)^*}&f^*\ar@{}[r]|{\cong}\ar@2[d]^(.4){\alpha^*}&
g^*\ar@2[ld]^{1_{g^*}}\\(g\!\circ 1_a)^*\ar@2[r]^-{\br^*}&g^*& }
\xymatrix{\\ =\\}
\xymatrix{&1_a^*f^*\ar@2[ld]_{\chi}\ar@2[d]\ar@{}@<-1pt>[d]^{1_a^*\alpha^*}
\ar@{}@<-3pt>[rd]^(.5){\overset{(\ref{4})}\cong} & f^*\ar@2[l]_(.4){\chi\! f^*}
\ar@2[d]^(.4){\alpha^*}\\
(f\!\circ\!1_a)^*\ar@{}@<-2pt>[r]|(.55){\Rrightarrow}\ar@{}@<4pt>[r]|(.55){\chi}\ar@2[d]_{(\alpha\circ
1)^*}&1_a^*g^* \ar@2[ld]_{\chi}
\ar@{}@<-3pt>[d]|(.55){\Rrightarrow}|(.38){\delta} &
g^*\ar@2[l]_(.4){\chi\! g^*} \ar@2[ld]^{1_{g^*}}\\(g\!\circ
1_a)^*\ar@2[r]^-{\br^*}&g^*& }
$$
$$
\xymatrix{&f^*1_b^*\ar@2[ld]_{\chi}
\ar@{}@<-3pt>[d]|(.55){\Rrightarrow}|(.38){\gamma}
& f^*\ar@2[l]_(.4){f^*\!\chi}\ar@2[ld]\ar@{}@<3pt>[ld]_(.4){1_{f^*}}
\ar@2[d]^(.4){\alpha^*}\\
(1_b\!\circ\!f)^*
\ar@{}@<30pt>[d]|(.56){\cong}|(.38){\xi}
\ar@2[r]^{\bl^*}\ar@2[d]_{(1\circ\alpha)^*}&f^*\ar@{}[r]|{\cong}\ar@2[d]^(.4){\alpha^*}&
g^*\ar@2[ld]^{1_{g^*}}\\(1_b\!\circ g)^*\ar@2[r]^{\bl^*}&g^*& }
\xymatrix{\\ =\\}
\xymatrix{&f^*1_b^*\ar@2[ld]_{\chi}\ar@2[d]\ar@{}@<-1pt>[d]^{\alpha^*1_b^*}
\ar@{}@<-3pt>[rd]^(.5){\overset{(\ref{4})}\cong} & f^*\ar@2[l]_(.4){f^*\!\chi}
\ar@2[d]^(.4){\alpha^*}\\
(1_b\!\circ\!f)^*\ar@{}@<-2pt>[r]|(.55){\Rrightarrow}\ar@{}@<4pt>[r]|(.55){\chi}\ar@2[d]_{(1\circ \alpha)^*}
&g^*1_b^* \ar@2[ld]_{\chi}
\ar@{}@<-3pt>[d]|(.55){\Rrightarrow}|(.38){\gamma} &
g^*\ar@2[l]_(.4){g^*\!\chi} \ar@2[ld]^{1_{g^*}}\\(1_b\!\circ
g)^*\ar@2[r]^{\bl^*}&g^*& }
$$

$(\mathbf{C8})$ for every pair of composable 1-cells
$\xymatrix@C=13pt{a\ar[r]^{f}&b\ar[r]^g&c}$, the following equation
holds:
$$
\xymatrix@R=18pt@C=10pt{ f^*1_b^*g^*
\ar@{}@<-17pt>[rr]|(.55){\Rrightarrow}
\ar@{}@<-10pt>[rr]|(.55){\gamma g^*}
\ar@2[d]_{f^*\chi}\ar@2[rd]^{\chi g^*}&&
f^*g^*
\ar@2[d]^{1_{f^*g^*}}\ar@2[ll]_{f^*\chi g^*}
\ar@{}@<28pt>[ddd]|{\textstyle =}                  \\
f^*(g\!\circ\!1_b)^*
\ar@{}@<-15pt>[r]|{\Rrightarrow}
\ar@{}@<-9pt>[r]|{\omega}
\ar@2[dd]_{\chi}
& (1_b\!\circ\!f)^*g^*
 \ar@2[d]^{\chi}
\ar@2[r]^(.6){\bl^*\!g^*}&
f^*g^*\ar@2[dd]^{\chi}                 \\
&(g\!\circ\!(1_b\!\circ\! f))^*
\ar@{}@<4pt>[r]|(.7){\Rrightarrow}
\ar@{}@<10pt>[r]|(.7){\chi}
\ar@2[rd]^{(1_g\circ\, \bl)^*}&                \\
((g\!\circ\!1_b)\!\circ\! f)^*
\ar@{}@<8pt>[rr]|{\cong}
\ar@{}@<15pt>[rr]|{\xi}
\ar@2[ru]^{\aso^*}
\ar@2[rr]_{(\br\circ 1_f)^*}& &(g\!\circ\!f)^* }
\xymatrix@R=14pt@C=28pt{
f^*1_b^*g^*\ar@2[dd]_{f^*\chi}&
& f^*g^*\ar@2[ll]_{f^*\chi g^*}
\ar@2@<3pt>[dd]^{1_{f^*g^*}}
\ar@/_1pc/@2[dd]\ar@{}@<-22pt>[dd]|{f^*\!1_{g^*}}
\ar@{}@<-5pt>[dd]|{\cong}                               \\
& &\\
f^*(g\!\circ\!1_b)^*
\ar@{}@<-28pt>[rr]|(.5){\Rrightarrow}
\ar@{}@<-22pt>[rr]|(.5){\chi}
 \ar@{}@<26pt>[r]|(.7){\Rrightarrow}
\ar@{}@<33pt>[r]|(.7){f^*\!\delta}
\ar@2[rr]^(.6){f^*\!\br^*}
\ar@2[dd]_{\chi}&&f^*g^*\ar@2[dd]^{\chi} \\
& ~ & \\
((g\!\circ\!1_b)\!\circ\! f)^*
\ar@2[rr]_{(\br\circ 1_f)^*}&
&(g\!\circ\!f)^* }
$$

A lax bidiagram of bicategories $\F:\B^{\mathrm{op}}\to\Bicat$ is
called a {\em pseudo bidiagram of bicategories} whenever each of the
pseudo transformations $\chi$, in $(\mathbf{D4})$ and
$(\mathbf{D5})$, is a pseudo equivalence; that is, regarding $\B$ as
a tricategory whose 3-cells are all identities, a trihomomorphism
$\F:\B^{^{\mathrm{op}}}\to \Bicat$ in the sense of
Gordon-Power-Street \cite[Definition 3.1]{g-p-s}.

\begin{example}\label{1213}{\em
If $\C$ is any small category viewed as a bicategory, then a lax
bidiagram of bicategories over $\C$, as above, in which the
deformations $\xi$ in $(\mathbf{D6})$ and $(\mathbf{D7})$, and
$\chi$ in $(\mathbf{D8})$, are all identities is the same thing as a
{\em lax diagram of bicategories} $\F:\C^{\mathrm{op}}\to \Bicat$ as
in \cite[\S 2.2]{ccg-2}.

For  instance, let $X$ be any topological space and let $\C(X)$
denote its poset of open subsets, regarded as a category. Then a
{\em fibered bicategory in bigroupoids above $X$} is a lax diagram
of bicategories
$$
\F:\C(X)^{\mathrm{op}}\to \Bicat,
$$
such that all the bicategories $\F_U$ are bigroupoids, that is,
bicategories whose 1-cells are invertible up to a 2-cell, and whose
2-cells are strictly invertible. In particular, when all the
bigroupoids $\F_U$ are strict, that is, 2-categories, and  all the
homomorphisms $f^*:\F_U\to \F_V$ associated to the inclusions of
open sets $f:V\hookrightarrow U$ are 2-functors, we have the notion
of {\em fibered $2$-category in $2$-groupoids above the space} $X$.
Thus, 2-stacks and 2-gerbes on spaces are relevant examples of lax
diagrams of bicategories (see e.g. Breen \cite[Definitions 6.1, 6.2,
and 6.3]{Breen}).

For another example, if $\mathcal{T}$ is any small tricategory, as
in \cite[Definition 2.2]{g-p-s}, then its {\em Grothendieck nerve}
$$
\mathrm{N}\mathcal{T}:\Delta^{\!{\mathrm{op}}}\to \Bicat,
$$
whose bicategory of   $p$-simplices is
$$
\mathrm{N}_p\mathcal{T} = \bigsqcup_{(x_0,\ldots,x_p)\in \mbox{\scriptsize Ob}\mathcal{T}^{p+1}}\hspace{-0.3cm}
\mathcal{T}(x_1,x_0)\times\mathcal{T}(x_2,x_1)\times\cdots\times\mathcal{T}(x_p,x_{p-1}),
$$
\cite[Theorem 3.3.1]{c-h} gives a striking example of a pseudo
diagram of bicategories.}
\end{example}

\begin{example}\label{exap2} {\em For any bicategory $\B$, a {\em lax bidiagram of
categories} over $\B$, that is, a lax bidiagram
$\F:\B^{\mathrm{op}}\to \Bicat$ in which every bicategory $\F_a$,
$a\in \Ob\B$, is a category (i.e., a bicategory where all the
2-cells are identities) is the same thing as a contravariant lax
functor $\F:\B^{\mathrm{op}}\to \Cat$ to the 2-category $\Cat$ of
small categories, functors, and natural transformations, since the
condition of all $\F_a$  being categories forces all the
modifications in $(\mathbf{D6})- (\mathbf{D10})$ to be identities.

For example, any object $b$ of a bicategory $\B$ defines a pseudo
bidiagram of categories \cite[Example 10]{street}
$$
\B(-,b):\B^{\mathrm{op}}\to \mathbf{Cat},
$$
which carries an object $x\in\Ob\B$ to the hom-category $\B(x,b)$, a
1-cell $g:x\to y$ to the functor $g^*:\B(y,b)\to\B(x,b)$ defined by
$$ \xymatrix@C=25pt {y \ar@/^0.8pc/[r]^{f} \ar@/_0.8pc/[r]_-{f'}
\ar@{}[r]|{\Downarrow\beta} & b} \xymatrix{\ar@{|->}[r]^{g^*}&} \xymatrix @C=40pt{x
\ar@/^0.8pc/[r]^{f\circ g} \ar@/_0.8pc/[r]_-{f'\!\circ g} \ar@{}[r]|{\Downarrow \beta\circ 1_{\!g}} &
b},
$$
and a 2-cell $\alpha:g\Rightarrow g'$ is carried to the natural
transformation $\alpha^*:g^*\Rightarrow g'^*$ that assigns to each
1-cell $f:y\to b$ in $\B$ the 2-cell $1_{\!f}\circ \alpha:f\circ
g\Rightarrow f\circ g'$. For $x\overset{g}\to y\overset{h}\to z$ any
two composable 1-cells of $\B$, the structure natural equivalence
$\chi: g^*h^*\cong (h\circ g)^*$, at any
 $f:z\to b$, is provided by the associativity constraint
$\boldsymbol{a}:(f\circ h)\circ g\cong f\circ (h\circ g)$, whereas
for any $x\in\Ob \B$, the structure natural equivalence
$\chi:1_{\B(x,b)}\cong 1_{x}^*$, at any $f:x\to b$, is the right
unit isomorphism $\br^{-1}:f\cong f\circ 1_{x}$.}
\end{example}

\section{The Grothendieck construction on lax bidiagrams of bicategories}\label{gt}
The well-known `Grothendieck construction', due to Grothendieck
\cite{groth, grothendieck} and Giraud \cite{giraud-2,giraud}, on
pseudo diagrams $(\F,\chi):\B^{\mathrm{op}}\to \Cat$ of small
categories with the shape of any given small category,  was
implicitly used in the proof given by Quillen of his famous Theorems
A and B for the classifying spaces of small categories
\cite{quillen}. Subsequently, since Thomason established his
celebrate Homotopy Colimit Theorem \cite{thomason}, the Grothendieck
construction has become an essential tool in homotopy theory of
classifying spaces.

In this section, our work is dedicated to extending the Grothendieck
construction to lax bidiagrams of bicategories $\F=(\F,
\chi,\xi,\omega,\gamma,\delta):\B^{\mathrm{op}}\to\Bicat$, where
$\B$ is any bicategory, since its use is a key for proving our main
results in the paper. But we are not claiming here much originality,
since extensions of the ubiquitous Grothendieck construction have
been developed in many general frameworks. In particular, we should
mention here three recent approaches to our construction: In
\cite{ccg-2}, Carrasco, Cegarra, and Garz\'on  study the
bicategorical Grothendieck construction on lax diagrams of
bicategories, as in Example \ref{1213}.  In
\cite{bakovic-2,bakovic}, Bakovi\'c performs the Grothendieck
construction on normal pseudo bidiagrams of bicategories, that is,
lax bidiagrams $\F$ whose modifications $\chi_b$ in $(\mathbf{D5})$
and  $\xi_f$  in $(\mathbf{D7})$ are identities, and whose pseudo
transformations $\chi_{g,f}$ in $(\mathbf{D4})$ are pseudo
equivalences. Buckley, in \cite{buckley}, presents the more general
case of pseudo bidiagrams, that is, when all the pseudo
transformations $\chi_{g,f}$ and $\chi_b$ in $(\mathbf{D4})$ and
$(\mathbf{D5})$ are pseudo equivalences.

The {\em Grothendieck construction on a lax bidiagram of
bicategories} $\F:\B^{\mathrm{op}} \to\Bicat$, as in $(\ref{lb})$,
assembles it into a  bicategory, denoted by
$$\xymatrix{\int_{\B}\!\F},$$
which is defined as follows:

\underline{The objects}\,  are pairs $(x,a)$, where $a\in\Ob\B$ and
$x\in \Ob\F_a$.

\underline{The 1-cells}\,  are pairs $(u,f):(x,a)\to (y,b)$,  where
$f:a\to b$ is a 1-cell in $\B$ and $u:x\to f^*y$ is a 1-cell in
$\F_a$.

\underline{The 2-cells}\, are pairs
$\xymatrix@C=25pt{(x,a)\ar@{}[r]|{\Downarrow
(\phi,\alpha)}\ar@/^0.8pc/[r]^{(u,f)}\ar@/_0.8pc/[r]_{(v,g)}&(y,b),}
$ consisting of a 2-cell $\xymatrix@C=0.5pc{a \ar@/^0.6pc/[rr]^{ f}
\ar@/_0.6pc/[rr]_{ g}\ar@{}[rr]|{\Downarrow\alpha} &
 &b }$ of $\B$ together with a 2-cell  $\phi:\alpha^*y\circ u\Rightarrow v$ in $\F_a$,
$$
\xymatrix@C=15pt@R=10pt{&f^*y\ar@{}@<3pt>[d]|(.6){\Downarrow\phi}\ar[rd]^{\alpha^*\!y}&\\
 x\ar[ru]^{u}\ar[rr]_{v}&&g^*y.}
$$

\underline{The vertical composition}  of 2-cells in $\int_\B\F$, $\xymatrix@C=25pt{(x,a) \ar@/^0.8pc/[r]^{(u,f)} \ar@/_0.8pc/[r]_{(v,g)}
\ar@{}[r]|{\Downarrow(\phi,\alpha)} & (y,b)}$ and $\xymatrix@C=25pt{(x,a) \ar@/^0.8pc/[r]^{(v,g)} \ar@/_0.8pc/[r]_{(w,h)}
\ar@{}[r]|{\Downarrow(\psi,\beta)} & (y,b)}$, is the 2-cell
$$
\xymatrix@C=50pt{(x,a)\ar@/^1pc/[r]^{(u,f)}\ar@{}[r]|(.5){\Downarrow (\psi\odot \phi,\beta\cdot \alpha)}
\ar@/_1pc/[r]_{(w,h)}&(y,b),}
$$
where $\beta\cdot\alpha$ is the vertical composition of $\beta$ with
$\alpha$ in $\B$,
 and $\psi\odot \phi:(\beta\cdot\alpha)^*\!y\circ u\Rightarrow w$ is the 2-cell of $\F_a$ obtained
  by pasting the diagram below.
$$
\xymatrix@C=40pt{\ar@{}@<-45pt>[dd]|(.3){\textstyle \psi\odot \phi:} & f^*y \ar[rdd]^{(\beta\cdot\alpha)^*y} \ar[d]|{\alpha^*y} &
	  \\
	  & g^*y \ar[rd]_{\beta^*y}  \ar@{}@<3pt>[r]|(.3){\overset{\xi}\cong}&
	  \\
	  x \ar[ruu]^u \ar[ru]_{v} \ar[rr]_w \ar@{}@<40pt>[r]|(.7){\Downarrow \phi}  \ar@{}@<20pt>[rr]|{\Downarrow \psi}&& h^*y
}
$$

The vertical composition of 2-cells so defined is associative and
unitary thanks to the coherence conditions $(\mathbf{C1})$ and
$(\mathbf{C2})$. The identity 2-cell, for each 1-cell
$(u,f):(x,a)\to (y,b)$, is $$\begin{array}{c}
1_{(u,f)}=(\overset{_\cdot}1_{(u,f)},1_f):(u,f)\Rightarrow (u,f).\\
\overset{\cdot}1_{(u,f)}=\big(\xymatrix{1_f^*y\circ
u\ar@2[r]^-{\xi^{-1}\circ 1}&1_{f^*\!y}\circ
u\overset{\bl}\Rightarrow u}\big)
\end{array}
$$

Hence, we have defined the hom-category
$\int_\B\F\big((x,a),(y,b)\big)$, for any two objects $(x,a)$ and
$(y,b)$ of $\int_\B\F$. Before continuing the description of this
bicategory, we shall do the following useful observation:
\begin{lemma}\label{isogro}
A $2$-cell $(\phi,\alpha):(u,f)\Rightarrow (v,g)$ in
$\int_\B\F\big((x,a),(y,b)\big)$ is an isomorphism if and only if
both $\alpha:f\Rightarrow g$, in $\B(a,b)$, and $\phi:\alpha^*y\circ
u\Rightarrow v$, in $\F_a(x,g^*y)$, are isomorphisms.
\end{lemma}
\begin{proof} It is quite straightforward, and we leave it to the reader.
\end{proof}

We return now to the description of the bicategory $\int_\B\F$.

\underline{The horizontal composition}  of two 1-cells
$\xymatrix@C=20pt{(x,a)\ar[r]^{(u,f)}&(y,b)\ar[r]^{(u',f')}&(z,c)}$
is the 1-cell
$$
(u',f')\circ (u,f)=(u'\circledcirc u,f'\circ f):(x,a)\longrightarrow (z,c),
$$
where $f'\circ f:a\to c$ is the composite in $\B$ of the 1-cells $f$
and $f'$, while $$u'\circledcirc u=\chi z\circ (f^*u'\circ
u):x\longrightarrow (f'\circ f)^*z$$ is the composite in $\F_a$  of
$
\xymatrix@C=15pt{x\ar[r]^-{u}&f^*y\ar[r]^{f^*u'}&f^*{f'}^*z\ar[r]^-{\chi
z}&(f'\circ f)^*z}$.

\underline{The horizontal composition} of 2-cells is defined by
$$
\xymatrix@C=30pt{(x,a)\ar@{}[r]|{\Downarrow (\phi,\alpha)}\ar@/^1pc/[r]^{(u,f)}
\ar@/_1pc/[r]_{(v,g)}&(y,b)
\ar@{}[r]|{\Downarrow (\phi',\alpha')}\ar@/^1pc/[r]^{(u',f')}\ar@/_1pc/[r]_{(v',g')}&(z,c)}
\ \overset{\circ}\mapsto
\xymatrix@C=55pt{(x,a)\ar@{}[r]|{\Downarrow (\phi'\circledcirc\phi,\alpha'\circ \alpha)}
\ar@/^1.2pc/[r]^{(u'\circledcirc u,f'\circ f)}\ar@/_1.2pc/[r]_{(v'\circledcirc v,g'\circ g)}&(z,c),
}
$$
where $\alpha'\circ\alpha$ is the horizontal composition in $\B$ of
$\alpha'$ with $\alpha$, and $\phi'\circledcirc\phi$ is the 2-cell
in $\F_a$ canonically obtained by pasting the diagram below.
$$
\xymatrix@C=40pt{\ar@{}@<-45pt>[dd]|(.2){\textstyle \phi'\circledcirc\phi:}  & f^*y
\ar[dd]^{\alpha^*\!y}
\ar[r]^{f^*u'}\ar[rd]\ar@{}@<1pt>[rd]_{f^*\!v'}
\ar@{}@<52pt>[d]|(.35){\Downarrow f^*\!\phi'}&
f^*f'^*z\ar[r]^{\chi z}\ar[d]^{f^*\!\alpha'^*\!z}&
(f'\circ f)^*z\ar[dd]^{(\alpha'\circ \alpha)^*z}
\\ x\ar@{}[r]|(.6){\Downarrow \phi}\ar[ru]^{u}\ar[rd]_{v}&&
f^*g'^*z\ar[d]^{\alpha^*\!g'^*\!z}
\ar@{}[r]^{\chi z}
\ar@{}@<-1pt>[r]|{\cong}
&\\
&g^*y
\ar@{}@<28pt>[r]|{\widehat{\alpha^*}}
\ar@{}@<21pt>[r]|{\cong}
\ar[r]_{g^*\!v'}&g^*g'^*z\ar[r]_{\chi z}&(g'\circ g)^*z}
$$

Owing to the coherence conditions $(\mathbf{C3})$ and
$(\mathbf{C4})$, the horizontal composition so defined truly gives,
for any three objects $(x,a),(y,b),(z,c)$ of $\int_{\B}\!\F$,  a
functor
$$\xymatrix{\int_{\B}\!\F((y,b),(z,c))\times
\int_{\B}\!\F((x,a),(y,b))\ar[r]^(.65){\circ}&
\int_{\B}\!\F((x,a),(z,c)).}$$

\underline{The structure associativity} isomorphism, for any three
composable morphisms $$(x,a)\overset{(u,f)}\longrightarrow
(y,b)\overset{(v,g)}\longrightarrow
(z,c)\overset{(w,h)}\longrightarrow  (t,d),$$
$$(\overset{_\circ}\aso,\aso):\big((w\circledcirc v)\circledcirc u,(h\circ g)\circ f \big)
\cong\big(w\circledcirc (v\circledcirc u), h\circ (g\circ f)
\big),$$
  is provided by the associativity constraint
$\aso:(h\circ g)\circ f\cong h\circ (g\circ f)$ of the bicategory
$\B$, together with the isomorphism in the bicategory $\F_a$ $$
\overset{_\circ}\aso:\aso^*t\circ ((w\circledcirc v)\circledcirc
u)\cong w\circledcirc (v\circledcirc u),$$ canonically obtained from
the 2-cell pasted of the diagram
$$
\xymatrix{
&& & f^*(h\circ g)^*t\ar[r]^{\chi t}& ((h\circ g)\circ f)^* t\ar[dd]^{\aso^* t}\\
x \ar@{}@<-16pt>[rr]|(.75){\cong }
\ar[r]^{u}\ar[rrd]_{v\circledcirc u}&
f^*y\ar[r]\ar@{}@<-2pt>[r]^{f^*\!v}\ar@<4pt>[rru]^{f^*\!(w\circledcirc v)}&
f^*g^*z\ar[r]\ar@{}@<-2pt>[r]^{f^*\!g^*\!w}
\ar[d]_{\chi z}
\ar@{}@<16pt>[r]|{\cong }&
f^*g^*h^*t\ar[u]_{f^*\!\chi t}\ar[d]^{\chi h^*\!t}
\ar@{}@<-3pt>[r]|(.55){\cong}
\ar@{}@<3pt>[r]|(.55){\omega t}&\\
&&(g\circ f)^*z\ar[r]_{(g\circ f)^*\!w}
\ar@{}@<24pt>[r]|{\widehat{\chi}}
\ar@{}@<17pt>[r]|{\cong}
&(g\circ f)^*h^*t\ar[r]_{\chi t}&(h\circ(g\circ f))^*t }
$$

By Lemma \ref{isogro}, these associativity 2-cells are actually
isomorphisms in $\xymatrix{\int_{\B}\!\F}$. Furthermore, they are
natural thanks to the coherence condition $(\mathbf{C5})$, while the
pentagon axiom for them holds because of condition $(\mathbf{C6})$.

\underline{The identity 1-cell}, for each object $(x,a)$ in
$\xymatrix{\int_{\B}\!\F}$, is provided by the pseudo transformation
$\chi_a:1_{\F_a}\Rightarrow 1_a^*$, by $$ 1_{(x,a)}=(\chi
x,1_a):(x,a)\to (x,a).$$

\underline{The left and right unit constraints}, for each morphism
$(u,f):(x,a)\to (y,b)$ in $\xymatrix{\int_{\B}\!\F}$,
$$\begin{array}{cc}
(\overset{_\circ}\bl,\bl):1_{(y,b)}\circ (u,f)\cong (u,f),&
(\overset{_\circ}\br,\br):(u,f)\circ 1_{(x,a)}\cong (u,f),
\end{array}
$$
are respectively given by the 2-cells $\bl:1_b\circ f\Rightarrow f$
and $\br:f\circ 1_a\Rightarrow f$ of $\B$, together with the 2-cells
in $\F_a$ obtained by pasting the diagrams below.
$$
\xymatrix@R=62pt{f^*y \ar@{}@<-25pt>[d]|(0.35){\textstyle \overset{_\circ}\bl: }\ar[rrd]\ar@{}@<3pt>[rrd]_{1_{f^*\!y}}\ar[r]^{f^*\!\chi\!y}&
f^*1_b^*y\ar[r]^{\chi y}&(1_b\circ f)^*y\ar[d]^{\bl^*\!y}\\
x\ar@{}@<18pt>[r]|{\cong}
\ar@{}@<24pt>[r]|{\bl }
\ar@{}@<52pt>[rr]|(.70){\cong}
\ar@{}@<58pt>[rr]|(.70){\gamma y}
\ar[u]^{u}\ar[rr]_{u}&&f^*y
}\hspace{0.5cm}
\xymatrix{
1_a^*x
\ar@{}@<-25pt>[dd]|(0.35){\textstyle \overset{_\circ}\br: }
\ar[r]^{1_a^*\!u}&1_a^*f^*y\ar[r]^{\chi y}&(f\circ 1_a)^*y\ar[dd]^{\br^*y}\\
\ar@{}@<12pt>[r]|{\cong}
\ar@{}@<20pt>[r]|{\widehat{\chi}}
& f^*y
\ar[u]\ar@{}@<2pt>[u]_{\chi f^*\!y}
\ar[rd]\ar@{}@<-3pt>[rd]^{1_{f^*\!y}}
\ar@{}@<12pt>[r]|{\cong}
\ar@{}@<20pt>[r]|{\delta y}&
\\
x
\ar@{}@<16pt>[rr]|{\cong}
\ar@{}@<22pt>[rr]|{\bl}
\ar[ru]\ar@{}@<-2pt>[ru]^{u}\ar[uu]^{\chi x}\ar[rr]_{u}&&f^*y}
$$

These unit constraints in $\xymatrix{\int_{\B}\!\F}$ are
isomorphisms by Lemma \ref{isogro}, natural due to coherence
condition $(\mathbf{C7})$, and the coherence triangle for them
follows from condition $(\mathbf{C8})$. Hence,
$\xymatrix{\int_{\B}\!\F}$ is actually a bicategory.

As a consequence of the above construction we obtain the following
equalities on lax bidiagram of bicategories, which is used many
times along the paper for several proofs:

\begin{lemma} \label{gdo}Let $\F=(\F,
\chi,\xi,\omega,\gamma,\delta):\B^{\mathrm{op}}\to\Bicat$ be a lax
bidiagram of bicategories. The equations on modifications below
hold.

$(i)$ For any object $a$ of $\B$,
$$
\xymatrix@C=45pt{1_a^*\ar@2[rdd]^{1}\ar@2[d]_{1_a^*\chi }\ar@{}@<-24pt>[r]|(.7){\cong}
           &1_{\mathcal{F}_a}\ar@2[dd]^{\chi }\ar@2[l]_{\chi }
\ar@{}@<30pt>|{\textstyle =}[dd]
           \\
                    1_a^*1_a^*\ar@2[d]_{\chi} \ar@{}@<-15pt>[r]|(.35){\Rrightarrow}
                    \ar@{}@<-9pt>[r]|(.35){\gamma} & \\
                    (1_a\circ 1_a)^*\ar@2[r]_{\bl_1^*=\br_1^*}
                      &1_a^*}\hspace{0.3cm}
\xymatrix@C=15pt{1_a^*\ar@2[d]_{1_a^*\chi }
           &\ar@{}@<-11pt>[d]|(.4){\overset{(\ref{4})}\cong}&1_{\mathcal{F}_a}\ar@2[dd]^{\chi }\ar@2[ll]_{\chi }
                                        \ar@2[ld]\ar@{}@<-6pt>[ld]|{\chi }\\
                    1_a^*1_a^*\ar@2[d]_{\chi} & 1^*_a\ar@2[l]_(.4){\chi 1_a^*}
                                                     \ar@2[rd]_{1}
                                     \ar@{}[r]|(.6){\cong}
                          \ar@{}@<20pt>[l]|(.35){\Rrightarrow}
                    \ar@{}@<14pt>[l]|(.35){\delta}&\\
                    (1_a\circ 1_a)^*\ar@2[rr]_{\bl_1^*=\br_1^*}
                      &&1_a^*}
$$

$(ii)$ for every pair of composable 1-cells
$\xymatrix@C=13pt{a\ar[r]^{f}&b\ar[r]^g&c}$ in $\B$,
$$
\xymatrix@R=18pt@C=16pt{ f^*g^*1_c^*
\ar@{}@<-15pt>[rr]|(.6){\overset{(\ref{4})}\cong}
\ar@2[d]_{f^*\chi}\ar@2[rd]^{\chi 1_c^*}&&
f^*g^*
\ar@2[d]^{\chi}
\ar@2[ll]_{f^*g^*\chi }
\ar@{}@<30pt>[ddd]|{\textstyle =} \\
f^*(1_c\!\circ\!g)^*
\ar@{}@<-18pt>[r]|{\Rrightarrow}
\ar@{}@<-12pt>[r]|{\omega}
\ar@2[dd]_{\chi}&
(g\!\circ\!f)^*1_c^*
\ar@2[d]^{\chi}
&  (g\!\circ\! f)^*
\ar@2[dd]^{1_{(g\circ f)^*}}
\ar@2[l]_(.45){(g\circ f)^*\chi }
\\
&(1_c\!\circ\!(g\!\circ\! f))^*
\ar@{}@<12pt>[r]|{\Rrightarrow}
\ar@{}@<18pt>[r]|{\gamma}
\ar@2[rd]^{\bl^*}&                  \\
((1_c\!\circ\!g)\!\circ\!f)^*
\ar@{}@<8pt>[rr]|{\cong}
\ar@{}@<15pt>[rr]|{\xi}
\ar@2[ru]^{\aso^*}
\ar@2[rr]_{(\bl\circ 1_f)^*}& &(g\!\circ\!f)^* }
\xymatrix@R=14pt@C=22pt{
f^*g^*1_c^*\ar@2[dd]_{f^*\chi}&
& f^*g^*\ar@2[ll]_{f^*g^*\chi }
\ar@2[dd]^{\chi} \ar@2[ldd]\ar@{}@<4pt>[ldd]_{f^*1_{g^*}}
\\
& &\\
f^*(1_c\!\circ\!g)^*
\ar@2[r]^{f^*\bl^*}
 \ar@{}@<28pt>[r]|(.75){\Rrightarrow}
\ar@{}@<34pt>[r]|(.75){f^*\gamma}
\ar@2[dd]_{\chi}&f^*g^*
\ar@2[rdd]^{\chi}\ar@{}[r]|{\cong}
&(g\!\circ\!f)^*\ar@2[dd]^{1_{(g\circ f)^*}}\\
& ~ & \\
((1_c\!\circ\!g)\!\circ\!f)^*
 \ar@{}@<18pt>[r]|(.8){\Rrightarrow}
\ar@{}@<24pt>[r]|(.8){\chi}
\ar@2[rr]_{(\bl\circ 1_f)^*}& &(g\!\circ\!f)^* }
$$
$$
\xymatrix@R=18pt@C=10pt{ 1_a^*f^*g^*
\ar@{}@<-15pt>[rr]|(.5){\Rrightarrow}
\ar@{}@<-8pt>[rr]|(.5){\delta g^*}
\ar@2[d]_{1_a^*\chi}\ar@2[rd]_{\chi g^*}&&
f^*g^*
\ar@2[d]^{1_{f^*\!g^*}}
\ar@2[ll]_{\chi  f^*g^*}
\ar@{}@<26pt>[ddd]|{\textstyle =} \\
1_a^*(g\circ\!f)^*
\ar@{}@<-16pt>[r]|{\Rrightarrow}
\ar@{}@<-11pt>[r]|{\omega}
\ar@2[dd]_{\chi}
& (f\!\circ\!1_a)^*\!g^*
\ar@2[r]^{\br^*\!g^*}
\ar@2[d]^{\chi}
&f^*g^*
\ar@2[dd]^{\chi}\\
&(g\!\circ\!(f\!\circ\! 1_a))^*
\ar@{}@<4pt>[r]|(.75){\Rrightarrow}
\ar@{}@<10pt>[r]|(.75){\chi}
\ar@2[rd]^(.6){(1_g\circ \br)^*}& \\
((g\!\circ\!f)\!\circ\!1_a)^*
\ar@{}@<8pt>[rr]|{\cong}
\ar@{}@<15pt>[rr]|{\xi}
\ar@2[ru]^{\aso^*}
\ar@2[rr]_{\br^*}& &(g\!\circ\!f)^* }
\xymatrix@R=14pt@C=20pt{
1_a^*f^*g^*\ar@2[dd]_{1_a^*\chi}&
& f^*g^*\ar@2[ll]_{\chi  f^*g^*}
\ar@2[dd]^{1_{f^*\!g^*}}
 \ar@2[ldd]_{\chi}
\\
& &\\
1_a^*(g\!\circ\!f)^*
 \ar@{}@<26pt>[r]|(.7){\overset{(\ref{4})}\cong}
\ar@2[dd]_{\chi}&(g\!\circ\!f)^*
\ar@2[l]_(.45){\chi (g\circ f)^*}
\ar@2[rdd]\ar@{}@<2pt>[rdd]_(.4){1_{(g\circ f)^*}}
\ar@{}[r]|{\cong}
&f^*g^*
\ar@2[dd]^{\chi}\\
& ~ & \\
((g\!\circ\!f)\!\circ\!1_a)^*
 \ar@{}@<23pt>[r]|(.7){\Rrightarrow}
\ar@{}@<30pt>[r]|(.7){\delta}
\ar@2[rr]_{\br^*}& &(g\!\circ\!f)^* }
$$
\end{lemma}
\begin{proof} $(i)$ follows from the equality $(\ref{equide})$ in the
bicategory $\int_\B\F$, that is,  $\br_{1_{(x,a)}}=\bl_{1_{(x,a)}}$,
for any $x\in \F_a$. Similarly, $(ii)$ is consequence of the
commutativity of triangles $(\ref{tri2})$ in $\int_\B\F$, for any
pair of composable 1-cells of the form
$$\xymatrix{(f^*g^*x,a)\ar[r]^{(1,f)}&(g^*x,b)\ar[r]^{(1,g)}&(x,c)},$$
for any $x\in\Ob\F_c$.
\end{proof}
\subsection{A cartesian square}
Let $\F=(\F,
\chi,\xi,\omega,\gamma,\delta):\B^{\mathrm{op}}\to\Bicat$ be any
given lax bidiagram of bicategories. For any bicategory $\A$ and any
lax functor $F:\A\to \B$, we shall denote by
\begin{equation}\label{f*}
\F F=(\F F, \chi_F,\xi_F,\omega_F,\gamma_F,\delta_F):\A^{\mathrm{op}}\to\Bicat
\end{equation}
the lax bidiagram of bicategories obtained by composing, in the
natural sense, $\F$ with $F$; that is, the lax bidiagram consisting
of the following data:

\vspace{0.2cm} $(\mathbf{D1})$ for each object $a$ in $\A$, the
bicategory \hspace{0.1cm}$\F_{Fa}$;

\vspace{0.2cm} $(\mathbf{D2})$ for each 1-cell $f:a\to b$ of $\A$,
the homomorphism \hspace{0.1cm}$(Ff)^*:\F_{Fb}\to \F_{Fa}$;

\vspace{0.2cm} $(\mathbf{D3})$ for each 2-cell $\xymatrix@C=0.5pc{a
\ar@/^0.7pc/[rr]^{ f} \ar@/_0.7pc/[rr]_{
g}\ar@{}[rr]|{\Downarrow\alpha} &
 &b }$ of $\A$, the pseudo transformation \hspace{0.1cm}$(F\alpha)^*:(Ff)^*\Rightarrow
 (Fg)^*$;

$(\mathbf{D4})$ for each  two composable 1-cells
$\xymatrix@C=13pt{a\ar[r]^{f}&b\ar[r]^{g}&c}$ in the bicategory
$\A$, the pseudo transformation \hspace{0.1cm} ${\chi_F}
_{_{\!g,f}}:(Ff)^*(Fg)^*\Rightarrow F(g\circ f)^*$ obtained by
pasting
$$\xymatrix@R=30pt{ &\F_{Fc}\ar[ld]\ar@{}@<4pt>[ld]_{(Fg)^*}
\ar@/_1.4pc/[rrd]|(.3){(Fg\circ Ff)^*}\ar[rrd]^{F(g\circ f)^*}& &\\
\F_{Fb}\ar@<-3pt>[rrr]_(.3){(Ff)^*}\ar@{}@<10pt>[r]|(.7){\Rightarrow}
\ar@{}@<17pt>[r]|(.7){\chi}&&\ar@{}@<8pt>[r]|(.1){\Rightarrow}
\ar@{}@<14pt>[r]|(.1){\widehat{F}^*}&\F_{Fa};}$$

$(\mathbf{D5})$ for each object $a$ of $\A$, the pseudo
transformation $
{\chi_{\!F}}_{\!_a}\!=\!\xymatrix@C=15pt{\big(1_{\F_{Fa}}\ar@2[r]^{\chi_{\!_{Fa}}}&
1_{Fa}^* \ar@2[r]^-{\widehat{F}^*_{a}}&F(1_a)^*\big);}$

$(\mathbf{D6})$ for any two vertically composable 2-cells
$\xymatrix@C=12pt{f\ar@2[r]^{\alpha}&g\ar@2[r]^{\beta}&h }$ in $\A$,
 the invertible modification \hspace{0.1cm}${\xi_F}_{_{\beta,\alpha}}=\xi_{_{F\beta,F\alpha}}\!:F(\beta)^*\circ F(\alpha)^*
 \Rrightarrow  F(\beta\cdot  \alpha)^*$;

\vspace{0.2cm}$(\mathbf{D7})$ for each 1-cell $f:a\to b$ of $\A$,
the invertible modification
\hspace{0.1cm}${\xi_F}_{_f}=\xi_{_{Ff}}\!:1_{{F(f)}^*}\Rrightarrow
1_{Ff}^*$;

$(\mathbf{D8})$ for every two horizontally composable 2-cells
$\xymatrix@C=0.5pc{a\ar@/^0.7pc/[rr]^f \ar@/_0.7pc/[rr]_{h}
\ar@{}[rr]|{\Downarrow\alpha}&  &
b\ar@/^0.7pc/[rr]^{g}\ar@/_0.7pc/[rr]_{k}
\ar@{}[rr]|{\Downarrow\beta}& & c}$ in $\A$,
$${\chi_F}_{_{\beta,\alpha}}:F(\beta\circ \alpha)^*\!\circ
{\chi_F}{_{_{g,f}}}\Rrightarrow{\chi_F}{_{_{k,h}}}\circ
 (F(\alpha)^*F(\beta)^*)$$
 is the invertible modification obtained by pasting the diagram
 below;
$$\xymatrix@R=40pt@C=40pt{F(f)^*F(g)^*\ar@2[r]^{F(\alpha)^*F(\beta)^*}
\ar@2[d]_{\chi}
\ar@{}@<-22pt>[r]|{\Rrightarrow}
\ar@{}@<-16pt>[r]|{\chi}
&
F(h)^*F(k)^*\ar@2[r]^{\chi}
& (Fk\circ Fh)^*\ar@2[d]^{\widehat{F}^*}\\
(Fg\circ Ff)^*\ar@2[r]_{\widehat{F}^*}
\ar@2[rru]\ar@{}@<-10pt>[rru]|(.6){(F\beta\circ F\alpha)^*}
&F(g\circ f)^*
\ar@{}@<16pt>[r]|(.6){\Rrightarrow}
\ar@{}@<23pt>[r]|(.6){\xi}
\ar@2[r]_{F(\beta\circ \alpha)^*}&F(k\circ h)^* }
$$
$(\mathbf{D9})$ for every three composable 1-cells
$\xymatrix@C=13pt{a\ar[r]^{f}&b\ar[r]^g&c\ar[r]^h&d}$ in $\A$, the
invertible modification $${\omega_F}_{_{h,g,f}}:\, F(\aso)^*\!\circ
({\chi_F}_{_{h\circ g,f}}\circ F(f)^*{\chi_F}_{_{h,g}} )\Rrightarrow
{\chi_F}_{_{h,g\circ f}}\circ {\chi_F}_{_{g,f}}F(h)^*$$ is obtained
from the modification pasted of the diagram below;
$$
\xymatrix@C=20pt{F(f)^*F(g)^*F(h)^*\ar@2[d]_{F(f)^*\chi}\ar@2[r]^{\chi\, F(h)^*}
\ar@{}@<-22pt>[rr]|(.3){\cong}
\ar@{}@<-16pt>[rr]|(.3){\omega}
&(Fg\circ Ff)^*F(h)^*
\ar@{}@<-22pt>[rr]|(.5){\cong}
\ar@{}@<-16pt>[rr]|(.5){\chi}
\ar@2[rd]^{\chi}\ar@2[r]^{\widehat{F}^*F(h)^*}&
F(g\circ f)^*F(h)^*\ar@2[rd]^{\chi}&\\
F(f)^*(Fh\circ Fg)^*
\ar@{}@<-22pt>[r]|(.45){\cong}
\ar@{}@<-16pt>[r]|(.45){\chi}
\ar@2[r]^{\chi} \ar@2[d]_{F(f)^*\widehat{F}^*} &((Fh\circ Fg)\circ
Ff)^*
\ar@{}@<-22pt>[rr]|(.5){\cong}
\ar@{}@<-15pt>[rr]|(.5){\xi}
\ar@2[r]^{\aso^*} \ar@2[d]_{(\widehat{F}\circ 1)^*} &
(Fh\circ(Fg\circ Ff))^* \ar@2[r]^{(1\circ \widehat{F})^*}&(Fh\circ
F(g\circ f))^* \ar@2[d]^{\widehat{F}^*}
\\
F(f)^*F(h\circ g)^*\ar@2[r]_{\chi}&(F(h\circ g)\circ Ff)^*
\ar@2[r]_{\widehat{F}^*}&F((h\circ g)\circ f)^*
\ar@2[r]_{F(\aso)^*}&F(h\circ (g\circ f))^* }
$$

$(\mathbf{D10})$ for any 1-cell $f:a\to b$ of $\A$, the invertible
modifications
$$\begin{array}{l}{\gamma_F}_{_{\!f}}:F(\bl_f)^*\circ ({\chi_F}_{_{1,f}}\circ
F(f)^*{\chi_F}_{_{\!b}})\Rrightarrow 1_{F(f)^*},\\[4pt]  {\delta_F}_{_{\!f}}:
F(\br_f)^*\circ ({\chi_F}_{_{f,1}}\circ
{\chi_F}_{_{\!a}}F(f)^*)\Rrightarrow 1_{F(f)^*},\end{array}$$ are,
respectively, canonically obtained from the modification pasted of
the diagrams below.
$$
\xymatrix@C=40pt{\ar@{}@<-55pt>[dd]|{\textstyle \gamma_F:}
F(f)^*F(1_b)^*\ar@2[d]_{\chi}
\ar@{}@<-23pt>[r]|{\cong}
\ar@{}@<-16pt>[r]|{\chi}
&
F(f)^*1_{Fb}^*\ar@2[d]^{\chi}\ar@2[l]_-{F(f)^*\widehat{F}^*}
\ar@{}@<-33pt>[r]|(.55){\cong}
\ar@{}@<-26pt>[r]|(.55){\gamma}
&F(f)^*
\ar@2[dd]^{1}\ar@2[l]_-{F(f)^*\chi}\\
(F1_b\circ Ff)^*\ar@2[d]_{\widehat{F}^*}
\ar@{}@<-23pt>[rr]|(.4){\cong}
\ar@{}@<-16pt>[rr]|(.4){\xi}
&(1_{Fb}\circ Ff)^*\ar@2[l]^-{(\widehat{F}\circ 1)^*}
\ar@2[rd]^{\bl^*}& \\
F(1_b\circ f)^*\ar@2[rr]_{F(\bl)^*}&&F(f)^*
}
$$
$$
\xymatrix@C=40pt{\ar@{}@<-55pt>[dd]|{\textstyle \delta_F:}
F(1_a)^*F(f)^*\ar@2[d]_{\chi}
\ar@{}@<-23pt>[r]|{\cong}
\ar@{}@<-16pt>[r]|{\chi}
&
1_{Fa}^*F(f)^*\ar@2[d]^{\chi}\ar@2[l]_-{\widehat{F}^*F(f)^*}
\ar@{}@<-33pt>[r]|(.55){\cong}
\ar@{}@<-26pt>[r]|(.55){\delta}
&F(f)^*
\ar@2[dd]^{1}\ar@2[l]_-{\chi F(f)^*}\\
(Ff\circ F1_a)^*\ar@2[d]_{\widehat{F}^*}
\ar@{}@<-23pt>[rr]|(.4){\cong}
\ar@{}@<-16pt>[rr]|(.4){\xi}
&(Ff\circ 1_{Fa})^*\ar@2[l]^-{(1\circ \widehat{F})^*}
\ar@2[rd]^{\br^*}& \\
F(f\circ 1_a)^*\ar@2[rr]_{F(\br)^*}&&F(f)^*
}
$$

There is an induced lax funtor
\begin{equation}\label{ilf}
\xymatrix{\bar{F}:\int_\A\F F \to \int_\B\F}
\end{equation}
given on cells by
$$
\xymatrix@C=30pt{(x,a)\ar@{}[r]|{\Downarrow (\phi,\alpha)}\ar@/^1pc/[r]^{(u,f)}\ar@/_1pc/[r]_{(v,g)}&(y,b)
}\  \overset{\bar{F}}\mapsto\  \xymatrix@C=30pt{(x,Fa)
\ar@{}[r]|{\Downarrow (\phi,F\alpha)}\ar@/^1pc/[r]^{(u,Ff)}\ar@/_1pc/[r]_{(v,Fg)}&(y,Fb),
 }
$$
and whose structure constraints are canonically given by those of
$F$, namely: For every two composable 1-cells
$\xymatrix@C=20pt{(x,a)\ar[r]^{(u,f)}&(y,b)\ar[r]^{(v,g)}&(z,c)}$ in
$\int_\A\F F$, the corresponding structure 2-cell of $\bar{F}$ for
their composition is
$$
 (\aso^{-1},\widehat{F}):\bar{F}(v,g)\circ \bar{F}(u,f)\cong \bar{F}((v,g)\circ (u,f)),
$$
where $\widehat{F}=\widehat{F}_{g,f}:Fg\circ Ff\Rightarrow F(g\circ
f)$ is the structure 2-cell of $F$, and
$$
\aso^{-1}: \widehat{F}_{g,f}^*\circ(\chi_{_{Fg,F\!f}}z\circ ( F(f)^*(v)\circ u))\cong
(\widehat{F}_{g,f}^*\circ \chi_{_{Fg,F\!f}}z)\circ (F(f)^*(v)\circ u)
$$
is the associativity isomorphism in the bicategory $\F_{Fa}$. For
$(x,a)$ any object  of the bicategory $\int_\A\F F$, the
corresponding structure 2-cell of $\bar{F}$ for its identity is
$$
(1,\widehat{F}):1_{\bar{F}(x,a)}\Rightarrow \bar{F}1_{(x,a)},
$$
where $\widehat{F}=\widehat{F}_a:1_{Fa}\Rightarrow F(1_a)$ is the
structure 2-cell of $F$, and $1$ is the is the identity 2-cell  of
the 1-cell $\widehat{F}_a^*x\circ \chi_{Fa}x:x\to F(1_a)^*x$ in the
bicategory $\F_{Fa}$.

Then, although the category of bicategories and lax functors has no
pullbacks in general, if, for any lax bidiagram of bicategories
$\F:\B^{\mathrm{op}}\to\Bicat$ as above, we denote by
\begin{equation}\label{proj2fun}
\xymatrix{P:\int_\B\F \to \B}
\end{equation}
the canonical projection 2-functor, which is defined by
$$\xymatrix@C=30pt{(x,a)\ar@{}[r]|{\Downarrow
(\phi,\alpha)}\ar@/^1pc/[r]^{(u,f)}\ar@/_1pc/[r]_{(v,g)}&(y,b)}\ \overset{P}\mapsto\
\xymatrix@C=25pt{a\ar@{}[r]|{\Downarrow \alpha}\ar@/^0.7pc/[r]^{f}
\ar@/_0.7pc/[r]_{g}&b,}
$$
 the following fact holds:
\begin{lemma}\label{lesqld} Let $\F:\B^{\mathrm{op}}\to \Bicat$ be a lax bidiagram of bicategories.
For any lax functor $F:\A\to \B$, the induced square
$$
\xymatrix{
\int_\A\F F\ar[d]_{P}\ar[r]^{\bar{F}}&\int_\B\F\ar[d]^{P}\\
\A\ar[r]^{F}&\B
}
$$
is cartesian in the category of bicategories and lax functors.
\end{lemma}
\begin{proof}
Any pair of lax functors, say $L:\C\to \A$ and $M:\C\to \int_\B\F$,
such that $FL=PM$ determines a unique lax functor $N:\C\to \int_\A\F
F$
$$
\xymatrix{ \C\ar[rdd]_L\ar[rrd]^M\ar@{.>}[rd]|N&&\\
&\int_\A\F F\ar[d]^(.4){P}\ar[r]^(.4){\bar{F}}&\int_\B\F\ar[d]^{P}\\
&\A\ar[r]^{F}&\B
}
$$
such that $PN=L$ and $\bar{F}N=M$, which is defined as follows:
Observe that the lax functor $M$ carries any object $a\in \Ob\C$ to
an object of $\int_\B\F$ which is necessarily written in the form
$Ma=(Da, FLa)$ for some object $Da$ of the bicategory $\F_{FLa}$.
Similarly, for any 1-cell $f:a\to b$ in $\C$, we have $Mf=(Df,FLf)$,
for some 1-cell $Df:Da\to FL(f)^*Db$ in $\F_{FLa}$, and, for any
2-cell $\alpha:f\Rightarrow g\in \C(a,b)$, we have
$M\alpha=(D\alpha,FL\alpha)$, for $D\alpha:FL(\alpha)^*Db\circ
Df\Rightarrow Dg$ a 2-cell in $\F_{FLa}$. Also, for any pair of
composable 1-cells $a\overset{f}\to b\overset{g}\to c$ and any
object $a$
 in $\C$, the structure
2-cells of $M$ can be respectively written in a similar form
  as
 $$\widehat{M}_{g,f}=(\widehat{D}_{g,f},F\widehat{L}_{_{g,f}}\circ \widehat{F}_{_{Lg,Lf}}):
 (Dg,FLg)\circ (Df,FLf)\Rightarrow (D(g\circ f),FL(g\circ f))$$
$$\widehat{M}_{a}=(\widehat{D}_{a},F(\widehat{L}_a)\circ \widehat{F}_{La}):1_{(Da,FLa)}\Rightarrow (D1_a,FL1_a),$$
for some 2-cells $\widehat{D}_{g,f}$ and $\widehat{D}_{a}$ of the
bicategory $\F_{FLa}$. Then, the claimed $N:\C\to \int_\A\F F$ is
the lax functor which acts on cells by
$$
\xymatrix@C=0.5pc{a
\ar@/^0.7pc/[rr]^{ f} \ar@/_0.7pc/[rr]_{
g}\ar@{}[rr]|{\Downarrow\alpha} &
 &b} \ \overset{N}\mapsto
\xymatrix@C=1.4pc{(Da,La)
\ar@/^1.2pc/[rr]^{ (Df,Lf)} \ar@/_1.2pc/[rr]_{(Dg,Lg)}\ar@{}[rr]|{\Downarrow(D\alpha,L\alpha)} &
 &(Db,Lb)}
$$
and whose respective structure 2-cells,
 for any pair of composable 1-cells $a\overset{f}\to b\overset{g}\to c$ and any object $a$
 in $\C$, are
 $$\widehat{N}_{g,f}=(\widehat{D}_{g,f},\widehat{L}_{_{g,f}}):
 (Dg,Lg)\circ (Df,Lf)\Rightarrow (D(g\circ f),L(g\circ f)),$$
 $$\widehat{N}_{a}=(\widehat{D}_{a},\widehat{L}_a):1_{(Da,La)}\Rightarrow (D1_a,L1_a).$$
\end{proof}

\begin{remark}{\em \label{reotlax} There exist different other `dual' notions of \emph{bidiagrams of bicategories}, depending on
the covariant or contravariant choices for $(\mathbf{D2})$ and
$(\mathbf{D3})$, and the direction of the pseudo transformations
$\chi$ in $(\mathbf{D4})$ and $(\mathbf{D5})$, but the results we
present about lax bidiagrams are similarly proved for the different
cases. For example, in a \emph{covariant oplax bidiagram of
bicategories} $\F:\B\to \Bicat$ the data in $(\mathbf{D2})$ are
specified with homomorphisms $f_*:\F_a\to \F_b$ for the $1$-cells
$f:a\to b$ of $\B$, while in $(\mathbf{D4})$, the pseudo
transformations are of the form $\chi_{g,f}:(g\circ f)_*\Rightarrow
g_*f_*$. The corresponding data in $(\mathbf{D5}), (\mathbf{D8}),
(\mathbf{D9})$ and $(\mathbf{D10})$ change in a natural way. The
Grothendieck construction on such a bidiagram, has now $1$-cells
$(u,f):(x,a)\to (y,b)$ given by $f:a\to b$ a $1$-cell in $\B$ and
$u:f_*x\to y$ a $1$-cell in $\F_b$. The $2$-cells
$(\phi,\alpha):(u,f)\Rightarrow (v,g)$ are now given by a $2$-cell
$\alpha: f\Rightarrow g$ in $\B$ and a $2$-cell $\phi:u\Rightarrow
v\circ \alpha_*x$. The compositions and constraints of this
bicategory are defined in the same way as in the contravariant lax
case.}\end{remark}

\section{The homotopy cartesian square induced by a lax bidiagram}
For the general background  on simplicial sets we mainly refer to
the book by Goerss and Jardine \cite{g-j}. In particular, we will
use the following result, which can be easily proved from the
discussion made in \cite[IV, 5.1]{g-j} and Quillen's Lemma
\cite[Lemma in page 14]{quillen} (or \cite[IV, Lemma 5.7]{g-j}):

\begin{lemma}\label{simlem} Let  $p: E\to B$ be an arbitrary simplicial map.
For any $n$-simplex $x\in B_n$, let $p^{-1}(x)$ be the simplicial
set defined by the pullback diagram
$$
\xymatrix{
p^{-1}(x)\ar[r]\ar[d]&E\ar[d]^{p}\\ \Delta[n]\ar[r]^{\Delta x}&B,
}
$$
where $\Delta[n]=\Delta(-,[n])$ is the standard simplicial
$n$-simplex, whose $m$-simplices are the maps $[m]\to [n]$ in the
simplicial category $\Delta$, and $\Delta x: \Delta[n]\to B$ denotes
the simplicial map such that $\Delta x(1_{[n]})=x$.

Suppose that, for every $n$-simplex  $x\in B_n$, and for any map
$\sigma:[m]\to [n]$ in the simplicial category,  the induced
simplicial map  $p^{-1}(\sigma^*x)\to p^{-1}(x)$
$$
\xymatrix@R=10pt{p^{-1}(\sigma^*x)\ar@{.>}[rd]\ar[rrd]\ar[dd]&& \\
&p^{-1}(x)\ar[dd]\ar[r] &E \ar[dd]^{p}\\
 \Delta[m]\ar[rd]_{\Delta\sigma}\ar[rrd]^(.35){\Delta(\sigma^*x)}|!{[ru];[rd]}\hole&& \\
 & \Delta[n]\ar[r]_{\Delta x}&B}
$$
gives a homotopy equivalence on geometric realizations
$|p^{-1}(\sigma^*x)|\simeq |p^{-1}(x)|$. Then, for each vertex $v\in
B_0$, the induced square of spaces
$$
\xymatrix{
|p^{-1}(v)|\ar[r]\ar[d]&|E|\ar[d]^{|p|}\\ pt\ar[r]^{|v|}&|B|
}
$$
is homotopy cartesian, that is, $|p^{-1}(v)|$ is homotopy equivalent
to the homotopy
 fiber of the  map $|p|:|E|\to |B|$ over the 0-cell $|v|$ of $|B|$.
\end{lemma}

Like categories, bicategories are closely related to spaces through
the classifying space construction. We shall recall briefly from
\cite[Theorem 6.1]{ccg-1} that the {\em classifying space} of a
(small) bicategory can be defined by means of several, but always
homotopy equivalent, simplicial and pseudo simplicial objects that
have been characteristically associated to it. For instance, the
classifying space $\BB \B$ of the bicategory $\B$ may be thought of
as
$$
\BB\B=|\Delta\B|,
$$
the geometric realization of its (non-unitary) {\em
 geometric nerve} \cite[Definition 4.3]{ccg-1}; that is, the
 simplicial set
$$
\Delta\B:\Delta^{\!^{\mathrm{op}}} \to \ \Set,\hspace{0.6cm}
[n]\mapsto \lfunc([n],\B),
$$ whose $n$-simplices are  all lax
functors ${\z:[n]\to \B}$. Here, the ordered sets
${[n]=\{0,1,\dots,n\}}$ are considered as categories with only one
morphism $(i,j):i\to j$ when $0\leq i\leq j\leq n$, so that a
non-decreasing map $[m]\rightarrow [n]$ is the same as a functor.
Hence, a geometric $n$-simplex of $\B$ is a list of cells of the
bicategory
$$
\z=(\z_i,\z_{i,j},\widehat{\z}_{i,j,k},\widehat{\z}_i)
$$
which is geometrically represented by a diagram in $\B$ with the
shape of the 2-skeleton of an oriented standard $n$-simplex, whose
faces are triangles
$$
\xymatrix{ & \z_j \ar@{}@<3pt>[d]|(.6){\Downarrow\,\widehat{\z}_{i,j,k}}\ar[dr]^{\z_{j,k}}             \\
                     \z_i   \ar[ur]^{\z_{i,j}} \ar[rr]_{\z_{i,k}} && \z_k}
$$
with objects $\z_i$ placed on the vertices, 1-cells
$\z_{i,j}:\z_i\rightarrow \z_j$ on the edges, and 2-cells
$\widehat{\z}_{i,j,k}:\z_{j,k}\circ \z_{i,j}\Rightarrow \z_{i,k}$ in
the inner, together with 2-cells $\widehat{\z}_i:
1_{\z_i}\Rightarrow \z_{i,i}$.  These data are required to satisfy
the condition  that, for $0\leq i\leq j\leq k\leq l\leq n$,  each
tetrahedron is commutative in the sense that
$$
\xymatrix@C=30pt@R=30pt{\z_i\ar[r]^{\z_{i,l}}\ar[rd]|{\z_{i,k}}\ar[d]_{\z_{i,j}}& \z_l
\ar@{}[ld]|(0.3){\widehat{\z}\Uparrow}
\ar@{}@<7ex>[d]|{\textstyle =}\\
\z_j \ar@{}@<8pt>[r]|(.3){\Rightarrow}
\ar@{}@<14pt>[r]|(.3){\widehat{\z}}
\ar[r]_{\z_{j,k}}&\z_k\ar[u]_{\z_{k,l}}}\hspace{0.4cm}
\xymatrix@C=30pt@R=30pt{\z_i\ar[r]^{\z_{i,l}}\ar[d]_{\z_{i,j}}
\ar@{}[rd]|(0.3){\widehat{\z}\Uparrow}& \z_l \\
\z_j
\ar@{}@<8pt>[r]|(.7){\Leftarrow}
\ar@{}@<14pt>[r]|(.7){\widehat{\z}}
\ar[r]_{\z_{j,k}}\ar[ru]|{\z_{j,l}}&\z_k\ar[u]_{\z_{k,l}} }
$$
and, moreover,
$$
\xymatrix{\ar@{}@<-7pt>[rr]^(.2){1_{\z_i}}\ar@{}@<20pt>[d]|(.53){\Rightarrow}|(.36){\widehat{\z}}&
\z_i\ar@/_0.9pc/[ld]\ar[rd]^{\z_{i,j}}
\ar@/^0.7pc/[ld]
\ar@{}@<8pt>[d]|(.7){\Rightarrow}|(.5){\widehat{\z}}
& \ar@{}@<10pt>[d]|(.5){\textstyle =}\ar@{}@<28pt>[d]|(.5){\textstyle \br_{\z_{i,j}}\,,}
\\\z_i\ar[rr]_{\z_{i,j}}\ar@{}@<0.5pt>[rr]^(.37){\z_{i,i}}&&\z_j}
\hspace{0.4cm}
\xymatrix{\ar@{}@<-7pt>[rr]^(.2){1_{\z_j}}\ar@{}@<20pt>[d]|(.53){\Rightarrow}|(.36){\widehat{\z}}&
\z_j\ar@{}@<8pt>[d]|(.7){\Rightarrow}|(.5){\widehat{\z}} &
\ar@{}@<10pt>[d]|(.5){\textstyle =}\ar@{}@<28pt>[d]|(.5){\textstyle
\bl_{\z_{i,j}}.}
\\ \z_j\ar@/_0.7pc/[ru] \ar@/^0.9pc/[ru]
\ar@{}@<0.7pt>[rr]^(.42){\z_{j,j}}&&\z_i\ar[lu]_{\z_{i,j}}
\ar[ll]^{\z_{i,j}}}
$$

If $\sigma:[m]\to [n]$ is any  map in $\Delta$, that is, a functor,
the induced $\sigma^*:\Delta\B_n\to\Delta\B_m$ carries any
${\z:[n]\to \B}$ to  ${\sigma^*\z=\z \sigma:[m]\to \B}$, the
composite lax functor of $\z$ with $\sigma$.

On a small category $\C$, viewed as a bicategory in which all
2-cells are identities, the geometric nerve construction $\Delta\C$
gives the usual Grothendieck nerve of the category \cite{groth},
since, for any integer $n\geq 0$, we have
$\lfunc([n],\C)=\mathrm{Func}([n],\C)$. Hence, the space
$\BB\C=|\Delta\C|$ of a category $\C$, is the usual classifying
space of the category, as considered by Quillen in \cite{quillen}.
   In particular, the geometric nerve of the category $[n]$ is precisely
$\Delta[n]$, the  standard simplicial $n$-simplex, so the notation
is not confusing. Furthermore, for any bicategory $\B$, the
simplicial map $\Delta\z:\Delta[n]\to\Delta\B$ defined by a
$n$-simplex $\z\in \Delta\B_n$,  that is, such that
$\Delta\z(1_{[n]})=\z$,
 is precisely the simplicial map obtained by taking geometric nerves
 on the lax functor $\z:[n]\to\B$.  Thus,
if $\sigma:[m]\to [n]$ is any  map in $\Delta$, then
$$\Delta(\sigma^*\z)=\Delta(\z\sigma)=\Delta\z\Delta\sigma:\Delta[m]\to \Delta\B.$$

The following fact, which is proved in \cite[Proposition
7.1]{ccg-1}, will be used repeatedly in our subsequent discussions:

\begin{lemma}\label{trans}
If $F,G:\A\to\B$ are two lax functors between bicategories, then any
lax or oplax transformation, $\varepsilon:F\Rightarrow G$,
canonically defines a homotopy $\BB\varepsilon:\BB F\simeq \BB G$
between the induced maps on classifying spaces $\BB F,\BB
G:\BB\A\to\BB\B$.
\end{lemma}

Suppose that $\F:\B^{\mathrm{op}}\to \Cat$ is a functor, where $\B$
is any small category, such that for every morphism $f:b\to c$ of
$\B$ the induced map $\BB f^*:\BB\F_c\to\BB\F_b$ is a homotopy
equivalence. Then, by Quillen's Lemma \cite[Lemma in page
14]{quillen}, the induced commutative square of spaces
$$
\xymatrix{\BB\F_a\ar[r] \ar[d]&\mathrm{hocolim}_\B\BB\F\ar[d]\\
pt\ar[r]&\BB\B}
$$
is homotopy cartesian. By Thomason's Homotopy Colimit Theorem
\cite{thomason}, there is a natural homotopy equivalence
$\mathrm{hocolim}_\B\BB\F \simeq \BB\int_\B\F$. Therefore, there is
a homotopy cartesian square
$$
\xymatrix{\BB\F_a\ar[r] \ar[d]&\BB\int_\B\F\ar[d]\\
pt\ar[r]&\BB\B.}
$$

 We are now ready to state and prove the
following important result in this paper, which generalizes the
result above, as well as the results in \cite[Theorem 7.4]{ccg-1}
and \cite[Theorem 4.3]{cegarra}:

\begin{theorem}\label{th1} Let $\F=(\F, \chi,\xi,\omega,\gamma,\delta):\B^{\mathrm{op}}\to\Bicat$ be any given lax bidiagram of bicategories. For any object $a\in \Ob\B$, there is a commutative square in $\Bicat$
\begin{equation}\label{squbicat1}
\begin{array}{c}\xymatrix{
\F_a\ar[r]^{J}\ar[d]&\int_\B\F\ar[d]^{P}\\ [0]\ar[r]^{a}&\B
}
\end{array}
\end{equation}
where $P$ is the projection $2$-functor  $(\ref{proj2fun})$, $a$
denotes the normal lax functor carrying $0$ to $a$, and $J$ is the
natural embedding homomorphism $(\ref{emj})$ described below, such
that, whenever each 1-cell $f:b\to c$ in $\B$ induces a homotopy
equivalence $\BB f^*\!:\BB\F_c\simeq \BB\F_b$, then  the square of
spaces induced on classifying spaces below is homotopy cartesian.
\begin{equation}\label{hsqubicat}
\begin{array}{c}\xymatrix{
\BB\F_a\ar[r]^-{\BB J}\ar[d]&\BB\!\int_\B\F\ar[d]^{\BB P}\\ pt \ar[r]^{\BB a}&\BB\B
}
\end{array}
\end{equation}
\end{theorem}
\begin{proof} This is divided into four parts.

\vspace{0.2cm} \noindent{\em Part 1.} Here we exhibit the embedding
homomorphism in the square $(\ref{squbicat1})$
\begin{equation}\label{emj}\xymatrix@C=18pt{J\!=\!J(\F,a):\ \F_a\ar[r]& \int_\B\F .}\end{equation}

It is defined on cells of $\F_a$ by
$$
\xymatrix{x\ar@/^0.8pc/[r]^u\ar@/_0.8pc/[r]_v
            \ar@{}[r]|{\Downarrow \phi}
                        &y}\
\overset{J}\mapsto \
\xymatrix@C=30pt{(x,a)\ar@/^1pc/[r]^{(\chi y\circ u,1_a)}
                              \ar@/_1pc/[r]_{(\chi y\circ v,1_a)}
                              \ar@{}[r]|{\Downarrow (\tilde{\phi},1)}
                        & (y,a)}
$$
where $\chi y\circ u$ is the 1-cell of $\F_a$ composite of
$\xymatrix@C=16pt{x\ar[r]^{u}&y\ar[r]^-{\chi_{_a}y}&1_a^*y}$,
$1=1_{1_a}$, the identity 2-cell in $\B$ of the identity 1-cell of
$a$, and $\tilde{\phi}$ is the $2$-cell given by the pasting in the
diagram below.
$$
\xymatrix{\ar@{}@<-22pt>[dd]|{\textstyle \tilde{\phi}:}&&1_a^*y\ar@/^0.7pc/[dd]^{1_{1_a}^*y}\ar@/_0.7pc/[dd]\ar@{}@<-10pt>[dd]|(.4){1}
                  \ar@{}[dd]|(.41){\xi}|(.5){\cong}\\
          x\ar@/^0.7pc/[r]^u\ar@/_0.7pc/[r]_v\ar@{}[r]|{\Downarrow\phi}
                      &y\ar[ru]^{\chi y}\ar[rd]_{\chi y}
                        \ar@{}[r]|{\cong}
                        \ar@{}@<7pt>[r]|{\bl}
                        &\\
                    &&1_a^*y}
$$
For $\xymatrix@C=12pt{x\ar[r]^u&y\ar[r]^v&z}$, two composable
1-cells in $\F_a$, the corresponding constraint $2$-cell for their
composition is $(\widehat{J},\bl):Jv\circ Ju\cong J(v\circ u)$,
where $\bl=\bl_{1_a}:1_a\circ 1_a\cong 1_a$, while
$\widehat{J}=\widehat{J}_{v,u}$ is the $2$-cell of $\F_a$ given by
pasting the diagram
$$
\xymatrix@C=20pt@R=16pt{x
\ar@{}@<-28pt>[dd]|{\textstyle\widehat{J}_{v,u}:}
\ar[r]^u\ar[rdd]_{v\circ u}\ar@{}@<22pt>[d]|(.7){=}
                  & y\ar[r]^{\chi y}\ar[dd]^v
                                    &1^*_ay\ar[rr]^{1^*_a(\chi z\circ v)}\ar[rd]
                                    \ar@{}@<-7pt>[rd]|{1^*_av}
                                           \ar@{}@<-10pt>[dd]|(.45){\cong}|(.35){\widehat{\chi}}
                                    &\ar@{}[d]|{\cong}
                                    &1^*_a1^*_az\ar[r]^{\chi z}
                                    \ar@{}@<20pt>[dd]|(.45){\cong}|(.35){\gamma}
                                    &(1_a\circ 1_a)^*z\ar[dd]^{\bl^*z}\\
                                &&&1^*_az\ar[rrd]^1\ar[ru]_{1_a^*\chi z}
                                         \ar@{}[d]|(.4){\bl}|(.6){\cong}&&\\
                                &z\ar[rru]^{\chi z}\ar[rrrr]^{\chi z}&
                                  &&&1^*_az,}
$$
and, for any object $x$ of $\F_a$,  the structure isomorphism for
its identity is $(\widehat{J},1):1_{Jx}\cong J(1_x)$, where
$1=1_{1_a}$, and  $\widehat{J}=\widehat{J}_{x}$ is provided by
pasting the diagram in $\F_a$ below.
$$
\xymatrix@C=3pc@R=3pc{x
\ar@{}@<17pt>[d]|{\cong}
\ar[r]^{\chi x}\ar[d]_{1_x}\ar@{}@<-35pt>[d]|{\textstyle
\widehat{J}_x:}
            &1_a^*x\ar[d]^{1_{1_a}^*x}\ar@/_1.5pc/[d]\ar@{}@<-20pt>[d]|(.4){1}\ar@{}@<-8pt>[d]|(.37){\xi}|(.5){\cong}
                               \\
                    x\ar[r]^{\chi x}
                      &1^*_ax}
$$

So defined, it is straightforward to verify that $J$ is functorial
on vertical composition of 2-cells in $\F_a$. The naturality of the
structure 2-cells $Jv\circ Ju\cong J(v\circ u)$ follows from the
coherence conditions in (\textbf{C1}) and (\textbf{C2}), whereas the
hexagon coherence condition for them is verified thanks to
conditions (\textbf{C1}), (\textbf{C2}), and (\textbf{C7}), and the
result in Lemma \ref{gdo}$(ii)$ relating $\gamma$ with $\omega$. As
for the other two coherence conditions, one amounts to the equality
in Lemma \ref{gdo}$(i)$, and the other is easily checked.

\vspace{0.2cm} \noindent{\em Part 2.} Let $\z:[n]\to \B$ be any
given geometric $n$-simplex of the bicategory, $n\geq 0$. Then, as
in $(\ref{f*})$, we have a composite lax bidiagram of bicategories
$\F\z:[n]\to \Bicat$. In this part of the proof, we prove that  the
homomorphism
$$
\xymatrix@C=18pt{J\!=\!J(\F\z,0)\!:\ \F_{\z_0}\ar[r]& \int_{[n]}\F\z}
$$
induces a homotopy equivalence on classifying spaces:
\begin{equation}\label{bemjn}\xymatrix{\BB
J\!:\BB\F_{\z_0}\,\simeq\,
\BB\! \int_{[n]}\F\z.}\end{equation} This is a direct consequence
of the following general observation:
\begin{lemma}\label{lemK} Suppose $\C$ is a small category with an initial
object $0$, and let us regard $\C$ as a bicategory whose $2$-cells
are all identities. Then, for any lax bidiagram of bicategories
$\LL:\C^{\mathrm{op}}\to\Bicat$,  the homomorphism
$J\!=\!J(\LL,0):\LL_0\to \int_{\C}\!\LL$ induces a homotopy
equivalence  on classifying spaces, $\BB J\!:\BB\LL_0\simeq
\BB\!\int_{\C}\LL$.
\end{lemma}
\begin{proof}
For any object $a\in\Ob\C$, let $0_a:0\to a$ denote the unique
morphism in $\C$ from the initial object to $a$. There is a
homomorphism
$$\xymatrix{K=K(\LL,0):\int_\C\LL\to
\LL_0,}$$ which carries any object $(x,a)$  to
$K(x,a)=0_a^*x$, the image of $x$ by the homomorphism
$0_a^*:\LL_a\to \LL_0$, a 1-cell $(u,f):(x,a)\to (y,b)$ to
$$K(u,f)=\big(\xymatrix@C=35pt{0_a^*x\ar[r]^{0_a^*u}&0_a^*f^*y\ar[r]^-{\chi_{_{f,0_a}}\!y}&
(f\circ 0_a)^*y=0_b^*y}\big),
$$
and a 2-cell $ \xymatrix@C=30pt{(x,a)\ar@{}[r]|{\Downarrow
(\phi,1)}\ar@/^0.8pc/[r]^{(u,f)}\ar@/_0.8pc/[r]_{(v,f)}&(y,b)} $ to
the 2-cell $K(\phi,1):K(u,f)\Rightarrow K(v,f)$ obtained by pasting
the diagram below, where
$(A)=\big(\xymatrix@C=18pt{1_{0^*_af^*y}\ar@2[r]^{\widehat{0^*_a}}&
0_a^*1_{f^*y}\ar@2[r]^{0_a^*\xi}&0_a^*1_f^*y}\big)$.
$$
\xymatrix{\ar@{}@<-45pt>[d]|{\textstyle K(\phi,1):}
&0_a^*f^*y\ar[rrd]^{\chi y}\ar[dd]|(.3){0_a^*1_f^*y}\ar@/^20pt/[dd]^{1}
\ar@{}@<10pt>[dd]|(.60){\cong}|(.52){(A)}
&\ar@{}[dd]|(.55){\cong}|(.45){\br}\\
0_a^*x\ar[ru]^{0_a^*u}\ar[rd]_{0_a^*v}\ar@{}[r]|(.55){\Downarrow\, 0_a^*\phi}
&&&0_b^*y\\
&0_a^*f^*y\ar[rru]_{\chi y}&}
$$
For each object $(x,a)$ of $\int_{\C}\LL$, the structure isomorphism
$\widehat{K}:1_{K(x,a)}\cong K1_{(x,a)}$ is
$$
\xymatrix{\ar@{}@<-45pt>[d]|(.3){\textstyle \widehat{K}:}
0_a^*x \ar[rr]^{0_a^*\chi x}\ar[dd]_1
\ar@{}@<27pt>[dd]|(.35){\cong}|(.25){\gamma}
&
\ar@{}@<20pt>[dd]|(.65){\cong}|(.57){\bl}
&0_a^*1_a^*x
\ar@/^3pc/[lldd]^{\chi x}\ar[ld]_{\chi x}\\
&0_a^*x\ar@/^10pt/[ld]^1\ar@/_10pt/[ld]|{1_{0_a}^*\!x}
\ar@{}@<-24pt>[d]|(.53){\cong\xi}
& \\
0_a^*x&&
}
$$
while the constraint
 $\widehat{K}:K(v,g)\circ K(u,f)\cong
K((v,g)\circ (u,f))$,  for each pair of composable 1-cells
$\xymatrix@C=1pc{(x,a)\ar[r]^{(u,f)} & (y,b) \ar[r]^{(v,g)} &(z,c)}$
of $\int_{\C}\LL$, is given by pasting in $\LL_0$ the diagram below.
$$
\xymatrix{
\ar@{}@<-35pt>[dd]|(.3){\textstyle \widehat{K}:}
0_a^*f^*y\ar[r]^{\chi y}\ar@{}@<20pt>[dd]|{=}&0_b^*y\ar@{}@<25pt>[d]|(.6){\cong}|(.4){\widehat{\chi}}
\ar[r]^{0_b^*v}&0_b^*g^*z\ar[rr]^{\chi z}
\ar@{}@<40pt>[dd]|(.42){\cong}|(.35){\omega}&\ar@{}@<32pt>[d]|(.60){\cong}|(.45){\xi} &0_c^*z\\
&0_a^*f^*y\ar[u]^{\chi y}\ar[r]_{0_a^*f^*v}\ar@{}@<25pt>[d]|{\cong}&
0_a^*f^*g^*z\ar[rd]_{0_a^*\chi z}
\ar[u]^{\chi g^*z}&0_c^*z\ar@/^10pt/[ru]|{1_{0_c}^*z}\ar@/_10pt/[ru]_1
\ar@{}@<32pt>[d]|(.40){\cong}|(.25){\bl}&\\
0_a^*x\ar[uu]^{0_a^*u}\ar[ru]_{0_a^*u}\ar[rrr]^{0_a^*(\chi z\circ (f^*v\circ u))}
&&&0_a^*(g\circ f)^*z\ar[u]_{\chi z}\ar@/_2pc/[ruu]_{\chi z}&}
$$

There are also two pseudo transformations $$\varepsilon: JK
\Rightarrow 1_{\int_{\C}\!\LL},\hspace{0.2cm}\eta: 1_{\LL_0}
\Rightarrow KJ,$$ which are defined as follows: The component of
$\varepsilon$ at an object $(x,a)$ of $\int_\C\LL$ is
$$
\varepsilon(x,a)=(1_{0_a^*x},0_a):(0_a^*x,0)\to (x,a),
$$
and its naturality component at a morphism $(u,f):(x,a)\to (y,b)$ is
$$
(\widehat{\varepsilon},1): \varepsilon(y,b)\circ JK(u,f)\cong (u,f)\circ \varepsilon (x,a),
$$
where $\widehat{\varepsilon}$ is the 2-cell of $\LL_0$  pasted of
the diagram below.
$$
\xymatrix{
0_a^*x\ar[dd]_1\ar[r]^{0_a^*u}&0_a^*f^*y\ar[r]^{\chi y}&0_b^*y\ar[r]^{\chi 0_b^*y}
\ar[rrrdd]_{1}
&1_0^*0_b^*y\ar[rr]^{1_0^*1_{0_b^*y}}_{\cong}\ar@/_12pt/[rr]|1
\ar@{}@<8pt>[dd]|(.2){\delta}|(.3){\cong}
\ar[rrd]_{\chi y}&&1_0^*0_b^*y\ar[d]^{\chi y}\ar@{}@<-12pt>[d]|(.55){\cong}|(.40){\br}\\
\ar@{}[rrrrr]|(.4){\cong}&&&&&0_b^*y\ar[d]\ar@{}@<-2pt>[d]^{1_{0_b}^*\!y}
\\
0_a^*x\ar[r]^{0_a^*u}&0_a^*f^*y\ar[rrrr]^{\chi y}&&&&0_b^*y
}
$$
The pseudo transformation $\eta: 1 \Rightarrow KJ$  assigns to each
object $x$ of the bicategory $\LL_0$ the 1-cell $ \eta x=\chi x:
x\to 1_0^*x$, while its naturality isomorphism at any 1-cell $u:x\to
y$,
$$\widehat{\eta}: \eta y\circ u \,\cong\,KJ(u)\circ \eta x,$$
 is obtained by pasting the diagram below.
$$
\xymatrix@C=4pc@R=25pt{x \ar[rrr]^u \ar[ddd]_{\chi x} \ar@{}@<35pt>[d]|(.43){\widehat{\chi}}|(.6){\cong}
       &&& y \ar[ddd]^{\chi y} \ar[lld]_{\chi y} \ar@{}@<-35pt>[d]|(.76){\bl}|(.9){\cong}
          \\
      & 1^*_0y \ar[dd]^{1^*_0 \chi y}  \ar@/^3pc/[rrdd]^1 \ar@{}@<35pt>[d]|(.43){\gamma}|(.6){\cong} & &
      \\
      &&
      1^*_0y \ar@/^1pc/[rd]^(.3){1^*_0y} \ar@/_1pc/[rd]|1 \ar@{}@<35pt>[d]|(.43){{\xi}}|(.6){\cong} \ar@{}[d]|{\cong}&
      \\
      1^*_0x \ar[r]_{1^*_0(\chi y \circ u)} \ar[ruu]^{1^*_0u} &
      1_0^*1_0^*y \ar[rr]_{\chi y} \ar[ru]^{\chi y} \ar@<-10pt>@{}[ul]|(.4){\cong} &&
      1^*_0y
}
$$

Hence, by Lemma \ref{trans}, there are induced homotopies
$\BB\varepsilon:\BB J\,\BB K =\BB(JK)\simeq \BB
1_{\int_{\C}\!\LL}=1_{\BB\!\int_{\C}\LL}$ and $\BB\eta:
1_{\BB\LL_0}= \BB 1_{\LL_0}\simeq \BB(KJ)=\BB K\,\BB J $, and it
follows that both maps $\BB J$ and $\BB K$ are actually  homotopy
equivalences.\end{proof}

\vspace{0.2cm} \noindent{\em Part 3.}  Let $\sigma:[m]\to [n]$ be a
map in the simplicial category. By Lemma \ref{lesqld}, for any
geometric $n$-simplex $\z:[n]\to \B$ of the bicategory $\B$, we have
the square
$$
\xymatrix{
\int_{[m]}\F\z\sigma\ar[r]^{\bar{\sigma}}\ar[d]_{P}&\int_{[n]}\F\z\ar[d]^{P}\\
[m]\ar[r]^{\sigma}&[n]
}$$
which is cartesian in the category of bicategories and lax functors.
This  part has the goal of proving that the lax functor
$\bar{\sigma}$ induces a homotopy equivalence on classifying spaces:
\begin{equation}\label{bemjnt}\xymatrix{\BB \bar{\sigma}\!:\
\BB\!\int_{[m]}\F\z\sigma\,\simeq\,
\BB\! \int_{[n]}\F\z.}\end{equation}

To do that,  let us consider the square of lax functors
$$
\xymatrix@C=55pt{
\F_{\z_{\sigma 0}}\ar[d]_{\z_{0,\sigma 0}^*}\ar[r]^-{J=J(\F\z\sigma,0)}
&\int_{[m]}\F\z\sigma\ar[d]^{\bar{\sigma}}
\\
\F_{\z_0}\ar[r]^-{J=J(\F\z,0)}&\int_{[n]}\F\z,
}
$$
where $\z_{0,\sigma 0}^*$ is the homomorphism attached by the lax
diagram $\F:\B^{\mathrm{op}}\to \Bicat$ to the 1-cell $\z_{0,\sigma
0}:\z_0\to\z_{\sigma 0}$ of $\B$, and the homomorphisms $J$ are
defined as in $(\ref{emj})$. This square is not commutative, but
there is a pseudo transformation $\theta: J\z_{0,\sigma
0}^*\Rightarrow \bar{\sigma}J$, whose component at any object $x$ of
$\F_{\z\sigma 0}$ is the 1-cell of $\int_{[n]}\F\z$
$$
\theta x=(1_{\z_{0,\sigma 0}^*x},(0,\sigma 0)):\,(\z_{0,\sigma 0}^*x,0)\to (x,\sigma 0),
$$
and whose naturality isomorphism, at any 1-cell $u:x\to y$
 in $\F_{\z\sigma 0}$, is
$$\widehat{\theta}_u=(\tilde{\theta},1_{\z_{0,\sigma 0}}):\theta y\circ J\z_{0,\sigma 0}^* u\cong \bar{\sigma} J u\circ \theta x,$$
where $\tilde{\theta}$ is given by pasting in $\F_{\z_0}$ the
diagram below.
$$
\xymatrix@C=50pt{
\z_{0,\sigma 0}^*x\ar[r]^{\z_{0,\sigma 0}^*u}\ar[ddd]_{1}
       \ar@{}@<25pt>[dd]|(.45){\cong}
  &\z_{0,\sigma 0}^*y\ar[r]^{\chi_\z \z_{0,\sigma 0}^*y}
                      \ar[ddd]|{\z_{0,\sigma 0}^*\chi_{\z} y}
                                        \ar@/^14pt/[rrddd]|{1}
                                        \ar@{}@<-25pt>[ddd]|(.65){\cong}
                                        \ar@{}@<50pt>[ddd]|(.5){\cong}|(.45){\gamma}
    &\z_{0,0}^*\z_{0,\sigma 0}^*y\ar[r]^{\z_{0,0}^*1}_{\cong}
                                 \ar@{}@<-13pt>[r]|(.6){1}
                                 \ar@/_10pt/[r]
                                                            \ar[rd]_{\chi_{\z} y}
                                                        \ar@{}@<20pt>[dd]|(.55){\cong}|(.45){\delta}
    &\z_{0,0}^*\z_{0,\sigma 0}^*y\ar[d]^{\chi_{\z} y}
                           \ar@{}@<-15pt>[d]|{\cong}
\\
&&&\z_{0,\sigma 0}^*y\ar[dd]^{1^*_{\z_{0,\sigma0}}y}
\\
&&\z_{0,\sigma0}^*y\ar@/^7pt/[rd]|(.38){1^*_{\z_{0,\sigma0}}y}
                   \ar@/_9pt/[rd]|(.3)1
                                     \ar@{}@<46pt>[d]|(.6){\cong}|(.47){\xi}
                                    \ar@{}[d]|{\cong}
\\
\z_{0,\sigma 0}^*x\ar[r]^-{\z^*_{0,\sigma0}(\chi_{\z}y\circ u)}
                  \ar[ruuu]^{\z_{0,\sigma0}^*u}
  & \z_{0,\sigma0}^*\z_{\sigma0,\sigma 0}^*y\ar[rr]^{\chi_\z y}\ar[ru]^{\chi_\z y}
    &&\z_{0,\sigma0}^*y
}
$$

Hence, by Lemma \ref{trans}, the induced square on classifying
spaces
$$
\xymatrix@C=35pt{
\BB\F_{\z_{\sigma 0}}\ar[d]_{\BB\z_{0,\sigma 0}^*}\ar[r]^-{\BB J}&
\BB\int_{[m]}\F\z\sigma\ar[d]^{\BB\bar{\sigma}}\\
\BB\F_{\z_0}\ar[r]^-{\BB J}&\BB\int_{[n]}\F\z
}
$$
is homotopy commutative. Moreover, both maps $\BB J$ in the square
are homotopy equivalences, as we showed in the proof of Lemma
\ref{lemK} above. Since, by hypothesis, the map $\BB \z_{0,\sigma
0}^*:\BB\F_{\z_{\sigma0}}\to \BB\F_{\z_{0}}$ is also a homotopy
equivalence, it follows that the remaining map in the square have
the same property, that is, the map $\BB\bar{\sigma}:
\xymatrix@C=15pt{ \BB\int_{[m]}\F\z\sigma\ar[r]& \BB\int_{[n]}\F\z}
$ is a homotopy equivalence.

\vspace{0.2cm} \noindent{\em Part 4.} Finally, we are ready to
complete here the proof of the theorem as follows: Let us consider
the induced simplicial map on geometric nerves $\Delta
P:\Delta\int_\B\F\to\Delta\B$. This verifies the hypothesis in Lemma
\ref{simlem}. In effect, thanks to Lemma \ref{lesqld}, for any
geometric $n$-simplex of $\B$, $\z:[n]\to \B$, the square
$$
\xymatrix{\int_{[n]}\F\z\ar[r]^{\bar{\z}}\ar[d]_{P} & \int_\B\F
\ar[d]^{P}\\
          [n]\ar[r]^\z&\B}
$$
is a pullback in the category of bicategories and lax functors,
whence the square induced by taking geometric nerves
$$
\xymatrix{\Delta\int_{[n]}\F \z\ar[r]^{\Delta\bar{\z}}\ar[d]_{\Delta P} &
\Delta\int_\B\F \ar[d]^{\Delta P}\\
          \Delta[n]\ar[r]^{\Delta\z}&\Delta\B}
$$
is a pullback in the category of simplicial sets. Thus,
$\xymatrix{\Delta P^{-1}(\Delta\z)=\Delta\int_{[n]}\F\z}$.

Furthermore, for any map $\sigma:[m]\to [n]$ in the simplicial
category, since the diagram of lax functors
$$
\xymatrix@R=10pt{\int_{[m]}\F\z\sigma\ar[rd]_{\bar{\sigma}}\ar[rrd]^{\overline{\z\sigma}}\ar[dd]_{P}&& \\
&\int_{[n]}\F\z\ar[dd]^{P}\ar[r]_{\bar{\z}} &\int_{\B}\F \ar[dd]^{P}\\
 [m]\ar[rd]_{\sigma}\ar[rrd]^(.35){\z\sigma}|!{[ru];[rd]}\hole&& \\
 & [n]\ar[r]_{\z}&\B}
$$
is commutative, the induced diagram of simplicial maps
$$
\xymatrix@R=10pt{\Delta
\int_{[m]}\F\z\sigma\ar[rd]_{\Delta\bar{\sigma}}\ar[rrd]^{\Delta\overline{\z\sigma}}\ar[dd]_{\Delta P}&& \\
&\Delta \int_{[n]}\F\z\ar[dd]^{\Delta P}\ar[r]_{\Delta \bar{\z}} &\Delta \int_{\B}\F \ar[dd]^{\Delta P}\\
 \Delta [m]\ar[rd]_{\Delta \sigma}\ar[rrd]^(.35){\Delta (\z\sigma)}|!{[ru];[rd]}\hole&& \\
 & \Delta [n]\ar[r]_{\Delta \z}&\Delta \B}
$$
is also commutative. Then, as $\sigma^*\z=\z\sigma$, the induced
simplicial map $\Delta P^{-1}(\sigma^*\z)\to \Delta P^{-1}(\z)$ is
precisely the map $\Delta\bar{\sigma}:\Delta \int_{[m]}\F\z\sigma\to
\Delta \int_{[n]}\F\z$, whose induced map on geometric realizations
is the homotopy equivalence $(\ref{bemjnt})$, $\xymatrix@C=15pt{\BB
\bar{\sigma}\!:\ \BB\!\int_{[m]}\F\z\sigma\, \simeq \, \BB\!
\int_{[n]}\F\z}$.

Hence, by Lemma \ref{simlem}, for each object $a\in\Ob\B$, the
square
$$\xymatrix{|\Delta P^{-1}(a)|\ar[r]\ar[d]&
|\Delta\int_\B\F| \ar[d]^{|\Delta P|}\ar@{}@<34pt>[d]|{\textstyle =}\\
          pt\ar[r]^-{|\Delta a|}&|\Delta \B|}
          \hspace{0.2cm}
\xymatrix{\BB\int_{[0]}\F a\ar[r]^{\BB\bar{a}}\ar[d] &
\BB\int_\B\F \ar[d]^{\BB P}\\
          \BB\Delta[0]\ar[r]^-{\BB a}&\BB\B}
$$
is homotopy cartesian. Furthermore, since the diagram of lax
functors
$$
\xymatrix@R=10pt@C=30pt{\F a\ar[rd]_{J(\F a,0)}\ar[rrd]^{J(\F,a)}\ar[dd]&& \\
&\int_{[0]}\F a\ar[dd]\ar[r]_{\bar{a}} &\int_{\B}\F \ar[dd]^{P}\\
 [0]\ar[rd]\ar[rrd]^(.35){a}|!{[ru];[rd]}\hole&& \\
 & [0]\ar[r]_{a}&\B}
$$
commutes, it follows that the square $(\ref{hsqubicat})$ is homotopy
cartesian square $(\ref{hsqubicat})$ as it is the composite of the
squares
$$
\xymatrix@C=40pt{
\BB\F_a\ar[r]^{\BB J(\F a,0)}\ar[d]&\BB\!\int_{[0]}\F a\ar[r]^{\BB \bar{a}}\ar[d]&\BB\!\int_\B\F\ar[d]^{\BB P}
\\ pt\ar[r]&pt \ar[r]^{\BB a}&\BB\B
}
$$
where  the map $\BB J(\F a,0)\!:\BB\F_a \simeq\BB\!\int_{[0]}\F a$
in the left square is one of the homotopy equivalences
$(\ref{bemjn})$, while the square on the right is homotopy
cartesian. \end{proof}

\section{The homotopy cartesian square induced by a lax functor}\label{theB}
In this section we prove the main theorem of this paper, that is, a
generalization to lax functors  (monoidal functors, for instance) of
the well-known Quillen's Theorem B \cite{quillen}. We shall first
extend Gray's construction \cite[Section 3.1]{gray2} of homotopy
fiber 2-categories to {\em homotopy fiber bicategories} of an
arbitrary lax functor between bicategories, so we can state the
corresponding `Theorem B' in terms of them.

Let $F:\A\to \B$ be any given lax functor between bicategories. As
in Example \ref{exap2}, each object $b$ of $\B$ gives rise to a
pseudo bidiagram of categories
$$
\B(-,b):\B^{\mathrm{op}}\to \Cat,
$$
which carries an object $x\in\Ob\B$ to the hom-category $\B(x,b)$,
and then also to the lax bidiagram of categories
\begin{equation}\label{cF-}
\B(-,b)F\!:\A^{\mathrm{op}}\to \Cat,
\end{equation}
obtained,  as in $(\ref{f*})$, by composing $\B(-,b)$ with $F$. The
Grothendieck construction on these lax bidiagrams leads to the
notions of {\em homotopy fiber} and {\em comma bicategories}:

\begin{definition} The {\em homotopy fiber}, $F\!\!\downarrow{\!_b}$,
 of a lax functor between bicategories $F:\A\to \B$ over an object $b\in\Ob\B$,
is the bicategory obtained as the Grothendieck construction on the
lax bidiagram $(\ref{cF-})$, that is,
$$
\xymatrix{F\!\!\downarrow\!_b=\int_\A\B(-,b)F.}
$$
In particular, when $F=1_\B$ is the identity functor on $\B$,
$$
\xymatrix{\B\!\!\downarrow{\!_b}=\int_\B\B(-,b)}
$$
is the {\em comma bicategory} of objects over $b$ of the bicategory
$\B$.
\end{definition}

It will be useful to develop here the Grothendieck construction,
exposed in Section \ref{gt}, in this particular case. Its objects
are pairs
\begin{equation}\label{obcomf} (f\!:Fa\to b,a)\end{equation}
with $a$ a 0-cell of $\A$ and $f$ a 1-cell of $\B$ whose source is
$Fa$ and target the fixed object $b$. The 1-cells
\begin{equation}\label{1celcomf}(\beta,u):(f,a)\to (f',a')\end{equation} consist of a $1$-cell $u:a\to a'$
in $\A$, together with a $2$-cell $\beta\!:f\Rightarrow f'\circ Fu$
in the bicategory $\B$,
$$
\xymatrix@C=10pt@R=14pt{ Fa\ar[rr]^{Fu}\ar[rd]_{f}& \ar@{}[d]|(.27){\beta}|(.48){\Rightarrow}&Fa'\ar[ld]^{f'}\\
 &b& }
$$
A 2-cell in $F\!\!\downarrow{\!_b}$,
\begin{equation}\label{2cellcommf}\xymatrix@C=30pt{(f,a)\ar@/^1pc/[r]^{(\beta,u)}
\ar@/_1pc/[r]_{(\beta',u')}
\ar@{}[r]|{\Downarrow\alpha}& (f',a')},
\end{equation}
 is a 2-cell $\alpha:u\Rightarrow u'$ in $\A$, such that the equation below holds in the category $\B(Fa,b)$.
\begin{equation}\label{1102}
\xymatrix@C=20pt@R=15pt{Fa\ar[rr]^{Fu'}_{\Uparrow F\alpha}
\ar[rdd]_{f}
\ar@/_1.1pc/[rr]
\ar@{}@<-17pt>[rr]|{Fu}&
&Fa'\ar[ldd]^{f'}\ar@{}@<24pt>[dd]|{\textstyle =}\\
&& \\
&b\ar@{}[uu]|(.35){\Rightarrow}|(.46){\beta}&
}
\hspace{0.3cm}
\xymatrix@C=20pt@R=15pt{Fa\ar[rr]^{Fu'}
\ar[rdd]_{f}
&
&Fa'\ar[ldd]^{f'}\\
&& \\
&b\ar@{}[uu]|(.55){\Rightarrow}|(.7){\beta'}&
}
\end{equation}

 Compositions, identities, and
the structure associativity and unit constraints in
$F\!\!\downarrow{\!_b}$ are as follows: For any given objects
$(f,a)$ and $(f',a')$ as in (\ref{obcomf}), the vertical composition
of 2-cells
$$
\xymatrix@C=45pt{(f,a)\ar[r]|{(\beta',u')}
\ar@/^1.5pc/[r]^{(\beta ,u)}\ar@/_1.5pc/[r]_{(\beta'',u'')}
\ar@{}[r]<10pt>|{\Downarrow\alpha}
\ar@{}[r]<-10pt>|{\Downarrow\alpha'}&(f',a')}\overset{\cdot}\mapsto
\xymatrix@C=40pt{(f,a)\ar@/^1pc/[r]^{(\beta,u)}
\ar@/_1pc/[r]_{(\beta'',u'')}\ar@{}[r]|{\Downarrow\alpha'\cdot\alpha}
&(f',a')}
$$
is given by the vertical composition $\alpha'\cdot\alpha$ of 2-cells
in $\A$.  The horizontal composition of two 1-cells in
$F\!\!\downarrow{\!_b}$,
$$
\xymatrix{(f,a)\ar[r]^{(\beta,u)}&(f',a')\ar[r]^{(\gamma,v)}&(f'',a'')}
$$
is the 1-cell
$$
(\gamma,v)\circ (\beta,u)=(\gamma \circledcirc \beta, v\circ u):
(f,a)\to (f'',a''),
$$
where the second component is the horizontal composition $v \circ u$
in $\A$, while the first one is the 2-cell in $\B$ obtained by
pasting the diagram below.
\begin{equation}\label{oo}\begin{array}{c}
\xymatrix@R=12pt{Fa\ar[r]^{Fu}\ar@/^1.6pc/[rr]^{F(v\circ u)}_{\widehat{F}\Uparrow}\ar[rdd]_{f}
\ar@{}@<32pt>[dd]|(.25){\beta}|(.38){\Rightarrow}
\ar@{}@<-35pt>[dd]|(.25){\textstyle \gamma \circledcirc \beta:}
&Fa'\ar[r]^{Fv}\ar[dd]|{f'}
\ar@{}@<16pt>[dd]|(.25){\gamma}|(.38){\Rightarrow}
&Fa''
\ar[ldd]^{f''}\\
&&\\
&b&
}\end{array}\end{equation}
The horizontal composition of 2-cells is simply given by the
horizontal composition of 2-cells in $\B$,
$$
\xymatrix@C=35pt{ (f,a)\ar@/^1pc/[r]^{(\beta,u)}
\ar@/_0.8pc/[r]_{(\beta',u')}
 \ar@{}[r]|{\Downarrow\alpha}&(f',a')\ar@/^1pc/[r]^{(\gamma,v)}
\ar@/_0.8pc/[r]_{(\gamma',v')}\ar@{}[r]|{\Downarrow \alpha'}&(f'',a'')}
\mapsto \xymatrix@C=40pt{(f,a)\ar@/^1pc/[r]^{(\gamma\circledcirc \beta, v\circ u)}
\ar@/_1pc/[r]_{(\gamma'\circledcirc \beta',v'\circ u')}
\ar@{}[r]|{\Downarrow \alpha'\circ \alpha}&(f'',a'')},
$$
and the  identity 1-cell of each 0-cell $(f:Fa\to b, a)$ is
$$\begin{array}{c}
1_{(f,a)}=(\overset{_\circ}{1}_{(f,a)},1_a ):(f,a)\to (f,a),
\\
\overset{_\circ}{1}_{(f,a)}=\big(\xymatrix{f\ar@2[r]^-{\br^{-1}}& f\circ 1_{Fa}\ar@2[r]^-{1_f\circ
\widehat{F}}&f\circ F(1_a)}\big)
\end{array}
$$

 Finally, the associativity, left and right unit constraints are obtained from those of $\A$
by the formulas
$$
\boldsymbol{a}_{(\beta'',u''),(\beta',u'),(\beta,u)}=\boldsymbol{a}_{u'',u',u},\hspace{0.3cm}
\boldsymbol{r}_{(\beta,u)}=
\boldsymbol{r}_u,\hspace{0.3cm}
\boldsymbol{l}_{(\beta,u)}=\boldsymbol{l}_u.
$$

We shall prove below that, under  reasonable necessary conditions,
the classifying spaces of the homotopy fiber bicategories $\BB
(F\!\!\downarrow{\!_b})$, of a lax functor $F:\A\to \B$,  realize
the homotopy fibers of the induced map on classifying spaces, $\BB
F:\BB\A\to \BB \B$. This fact will justify the name of `homotopy
fiber bicategories' for them. As a first step to do it, we state the
following particular case, when $F=1_\B$ is the identity
homomorphism:

\begin{lemma}\label{cont}
For any object $b$ of a bicategory $\B$, the classifying space of
the comma bicategory $\B\!\!\downarrow{\!_b}$ is contractible, that
is, $\BB(\B\!\!\downarrow{\!_b})\simeq pt$.
\end{lemma}
\begin{proof} Let $[0]\to \B\!\!\downarrow{\!_b}$ denote the
normal lax functor that carries $0$ to the object $(1_b,b)$, and let
$\mathrm{Ct}:\B\!\!\downarrow{\!_b}\to\B\!\!\downarrow{\!_b}$ be the
composite of $\B\!\!\downarrow{\!_b}\to [0] \to
\B\!\!\downarrow{\!_b}$. Then, the induced map on classifying spaces
$$\xymatrix@C=20pt{\BB(\B\!\!\downarrow{\!_b})\ar[r]^{\BB\mathrm{Ct}}&\BB(\B\!\!\downarrow{\!_b})}=
\xymatrix@C=15pt{\BB(\B\!\!\downarrow{\!_b})\ar[r]&\BB[0]=pt\ar[r]&\BB(\B\!\!\downarrow{\!_b})}$$
is a constant map. Now, let us observe that there is a canonical
oplax transformation $1_{\B\downarrow{_b}}\Rightarrow \mathrm{Ct}$,
whose component at any object $(f:a\to b,a)$ is the 1-cell
$(\bl^{-1}_f,f):(f,a)\to (1_b,b)$, and whose naturality component at
a 1-cell $(\beta,u):(f,a)\to (f',a')$ is
$$
\xymatrix{(f,a) \ar[r]^{(\bl^{-1},f)} \ar[d]_{(\beta,u)} \ar@{}@<25pt>[d]|(.39){\beta\cdot\bl}|(.51){\Leftarrow}
        & (1_b,b) \ar[d]^{(\br^{-1},1_b)}
      \\
      (f',a') \ar[r]_{(\bl^{-1},f')} &
      (1_b,b).
}
$$
This oplax transformation gives, thanks to Lemma $\ref{trans}$, a
homotopy between
$\BB(1_{\B\downarrow{_b}})=1_{\BB(\B\downarrow{_b})}$ and the
constant map $\BB\mathrm{Ct}$, and so we obtain the result.
\end{proof}

\begin{example}\label{ombi}{\em  Let $\B$ be a bicategory, and suppose $b\in\Ob\B$ is an object such
that the induced maps $\BB p^*:\BB\B(y,b)\to \BB\B(x,b)$ are
homotopy equivalences for the different morphisms $p:x\to y$ in $\B$
(for instance, any object of a bigroupoid). By Theorem \ref{th1}, we
have the fiber sequence $$ \BB \B(b,b)\to \BB
\B\!\!\downarrow{\!_{b}}\to \BB\B$$ in which the space $\BB
\B\!\!\downarrow{\!_{b}}$ is contractible by Lemma \ref{cont}.
Hence, we conclude the existence of a homotopy equivalence
\begin{equation}\label{loo1}
\Omega(\BB \B,\BB b)\simeq \BB(\B(b,b))
\end{equation}
between the loop space of the classifying space of the bicategory
with base point $\BB b$ and the classifying space of the category of
endomorphisms of $b$ in $\B$.

 The homotopy equivalence above is already known when the bicategory is strict,
 that is, when $\B$ is a 2-category. It appears as a main result in the paper by Del
 Hoyo \cite[Theorem 8.5]{Hoyo}, and it was also stated at the same time by Cegarra in
 \cite[Example 4.4]{cegarra}. Indeed, that homotopy equivalence $(\ref{loo1})$, for the case when $\B$ is a 2-category,
can be deduced from a result by Tillman about simplicial categories
in \cite[Lemma 3.3]{Tillmann}.}
\end{example}

Returning to an arbitrary  lax functor $F:\A\to\B$, we shall now pay
attention to two constructions with fiber homotopy bicategories.
First, we have that any 1-cell $p:b\to b'$ in $\B$ determines a
2-functor
\begin{equation}\label{h_*}
p_*\!:\,F\!\!\downarrow{\!_b}\to F\!\!\downarrow{\!_{b'}}
\end{equation}
whose function on objects is defined by
$$
p_*(Fa\overset{f}\to b,a)=(Fa\overset{p\circ f}\longrightarrow b',a).
$$
A 1-cell $(\beta,u):(f,a)\to (f',a')$ of $\B\!\!\downarrow{\!_b}$,
as in $(\ref{1celcomf})$ , is carried  to the 1-cell of
$\B\!\!\downarrow{\!_{b'}}$
$$\begin{array}{c}
p_*(\beta, u)=(p\circledcirc \beta, u):(p\circ f,a)\to(p\circ f',a'),
\\
p\circledcirc \beta=\big( \xymatrix@C=15pt{p\circ
f\ar@2[r]^-{1_p\circ \beta} & p\circ(f'\circ Fu)
\ar@2[r]^{\aso^{-1}}& (p\circ f')\circ Fu}\big)\end{array}$$ while, for
$\alpha:(\beta,u)\Rightarrow (\beta',u')$ any 2-cell in
$\B\!\!\downarrow{\!_b}$ as in $(\ref{2cellcommf})$,
$$
p_*(\alpha)=\alpha:(p\circledcirc \beta, u)\Rightarrow (p\circledcirc \beta', u').
$$
Secondly, by Lemma \ref{lesqld}, we have a pullback square in the
category of bicategories and lax functors  for any  $b\in \Ob \B$
\begin{equation}\label{sq1}
\begin{array}{c}
\xymatrix{F\!\!\downarrow{\!_b}\ar[d]_{P}\ar[r]^{\bar{F}}&\B\!\!\downarrow{\!_b}\ar[d]^{P}
\ar@{}@<25pt>[d]|{\textstyle =}
\\
\A\ar[r]^{F}&\B}\hspace{0.3cm}\xymatrix{
\int_\A \B(-,b)F\ar[d]_{P}\ar[r]^{\bar{F}}&\int_\B\B(-,b)\ar[d]^{P}\\
\A\ar[r]^{F}&\B}\end{array}
\end{equation}
where, recall, the 2-functors $P$ are the canonical projections
$(\ref{proj2fun})$, and $\bar{F}$ is the induced lax functor
$(\ref{ilf})$, which acts on cells by
$$
\xymatrix@C=30pt{(f,a)\ar@/^0.8pc/[r]^{(\beta,u)}
\ar@/_0.8pc/[r]_{(\beta',u')}
\ar@{}[r]|{\Downarrow\alpha}& (f',a')}\  \overset{\bar{F}}\mapsto\
\xymatrix@C=30pt{(f,Fa)\ar@/^0.8pc/[r]^{(\beta,Fu)}
\ar@/_0.8pc/[r]_{(\beta',Fu')}
\ar@{}[r]|{\Downarrow F\!\alpha}& (f',Fa'),}
$$
and whose structure constraints are canonically given by those of
$F$.

We are now ready to state and prove the following theorem, which is
just the well-known Quillen's Theorem B \cite{quillen} when the lax
functor $F$ in the hypothesis is an ordinary functor between
categories. The result therein also generalizes a similar result by
Cegarra \cite[Theorem 3.2]{cegarra}, which was stated for the case
when $F$ is a 2-functor between 2-categories, but the extension to
arbitrary lax functors between bicategories is highly nontrivial and
the proof we give here uses different tools.

\begin{theorem}\label{B} Let $F:\A\to\B$ be a lax functor between bicategories.  The following statements
are equivalent:

$(i)$ For every 1-cell $p:b\to b'$ in $\B$, the induced map $\BB
p_*:\BB(F\!\!\downarrow{\!_b})\to \BB(F\!\!\downarrow{\!_{b'}})$ is
a homotopy equivalence.

$(ii)$ For every object $b$ of $\B$, the induced square by
$(\ref{sq1})$ on classifying spaces
\begin{equation}\label{hs2}
\begin{array}{c}
\xymatrix{\BB(F\!\!\downarrow{\!_b})\ar[d]_{\BB P}\ar[r]^{\BB\bar{F}}&\BB(\B\!\!\downarrow{\!_b})
\ar[d]^{\BB P}
\\
\BB\A\ar[r]^{\BB F}&\BB\B}\end{array}
\end{equation}
is homotopy cartesian.

Therefore, in such a case, for each object $a\in\Ob\A$ such that
$Fa=b$, there is a homotopy fibre sequence
$$\BB(F\!\!\downarrow{\!_b})\to \BB \A\to\BB\B,$$ relative to the
base 0-cells $\BB a$ of $\BB \A$, $\BB b$ of $\BB \B$ and $\BB
(1_b,a)$ of $\BB(F\!\!\downarrow{\!_b})$, that induces a long exact
sequence on homotopy groups
$$\cdots \to \pi_{n+1}\BB\B\to\pi_n\BB(F\!\!\downarrow{\!_b})\to\pi_n\BB\A\to\pi_n\BB\B\to\cdots.$$
\end{theorem}
\begin{proof} $(ii)\Rightarrow (i)$ Suppose that $p:b\to b'$ is any 1-cell of $\B$.
Then, taking $\z:[1]\to \B$   the normal lax functor such that
$\z_{0,1}=p$, we have the
 path $\BB \z:\BB[1]\!=\!I\to \BB\B$,
    whose  origen is the point $\BB a$ and whose end is $\BB b$ (actually, $\BB\B$ is a CW-complex and $\BB\z$
    is one of its 1-cells).
 Since the homotopy fibers of a continuous map whose over points are connected by a path are homotopy equivalent, the result follows.

$(i)\Rightarrow (ii)$  This is divided into three parts.

{\em Part 1.} We begin here by noting that the bicategorical
homotopy fiber construction is actually the function on objects of a
covariant oplax bidiagram of bicategories
$$
F\!\!\downarrow \,\,=(F\!\!\downarrow\,, \chi,\xi,\omega,\gamma,\delta):\B\to\Bicat
$$
consisting of the following data:

\vspace{0.2cm} $(\mathbf{D1})$ for each object $b$ in $\B$, the
homotopy fiber bicategory \hspace{0.1cm}$F\!\!\downarrow{\!_b}$;

\vspace{0.2cm} $(\mathbf{D2})$ for each 1-cell $p:b\to b'$ of $\B$,
the 2-functor  $p_*\!:\, F\!\!\downarrow{\!_b}\to
F\!\!\downarrow{\!_{b'}}$ in $(\ref{h_*})$;

\vspace{0.2cm} $(\mathbf{D3})$ for each 2-cell $\xymatrix@C=0.5pc{b
\ar@/^0.6pc/[rr]^{p} \ar@/_0.6pc/[rr]_{
p'}\ar@{}[rr]|{\Downarrow\sigma} &
 &b' }$ of $\B$, the pseudo transformation $\sigma_*:p_*\Rightarrow
 p'_*$, whose component at an object $(f,a)$ of $F\!\!\downarrow{\!_{b}}$,
is the 1-cell $$\begin{array}{c}\sigma_*(f,a)=(\sigma \circledcirc
f,1_a):(p\circ f,a)\to (p'\circ f,a),\\
\sigma \circledcirc f=\big(\xymatrix@C=15pt{ p\circ
f\ar@2[r]^{\sigma\circ 1}& p'\circ f\ar@2[r]^-{\br^{-1}}
            & (p'\circ f)\circ 1_{Fa}\ar@2[r]^(.45){1\circ \widehat{F}} &
             (p'\circ f)\circ F1_a}\big)
\end{array}
$$
and whose naturality component at any 1-cell $(\beta,u):(f,a)\to
(f',a')$, as in $(\ref{1celcomf})$, is the canonical isomorphism
$\br^{-1}\cdot \bl:1_{a'}\circ u\cong u\circ 1_a$;
$$
\xymatrix@C=40pt{(p\circ f,a)\ar[r]^{(\sigma\circledcirc f,1_a)}
\ar[d]_{(p \circledcirc
\beta,u)}\ar@{}@<45pt>[d]|(.38){\br^{-1}\cdot\, \bl}|(.51){\cong}
& (p'\circ f,a)\ar[d]^{(p'\circledcirc \beta,u)} \\
(p\circ f',a')\ar[r]_{(\sigma\circledcirc f',1_{a'})}&(p'\circ
f',a') }
$$

$(\mathbf{D4})$ for each  two composable 1-cells
$\xymatrix@C=11pt{b\ar[r]^{p}&b'\ar[r]^{p'}&b''}$ in the bicategory
$\B$, the pseudo transformation ${\chi} _{_{\!p',p}}:(p'\circ
p)_*\Rightarrow p'_*p_*$ has component, at an object $(f,a)$ of
$F\!\!\downarrow{\!_{b}}$,  the 1-cell
$$\begin{array}{c}
(\text{{\aa}},1_a):((p'\circ p)\circ f,a)\to (p'\circ(p\circ f),a) ,\\
\text{{\aa}}=\big(\xymatrix@C=15pt{(p'\circ p)\circ f\ar@2[r]^{\aso}& p'\circ (p\circ f)\ar@2[r]^-{\br^{-1}}
            & (p'\circ (p\circ f))\circ 1_{Fa}\ar@2[r]^(.45){1\circ \widehat{F}} &
             (p'\circ (p\circ f))\circ F1_a}\big)
             \end{array}
$$
and whose naturality component at a 1-cell $(\beta,u):(f,a)\to
(f',a')$, is
$$
\xymatrix@C=40pt{((p'\circ p)\circ f,a)\ar[r]^{(\text{{\aa}},1)}
\ar[d]_{((p'\circ p) \circledcirc \beta),u)}
\ar@{}@<55pt>[d]|(.38){\br^{-1}\cdot\, \bl}|(.51){\cong}
& (p'\circ (p\circ f),a)\ar[d]^{(p'\circledcirc (p\circledcirc \beta),u)} \\
((p'\circ p)\circ f',a')\ar[r]_{(\text{{\aa}},1)}&(p'\circ (p\circ f'),a');
}
$$

$(\mathbf{D5})$ for each object $b$ of $\B$,
\hspace{0.1cm}$\chi_{_b}:{1_{b}}_*\Rightarrow 1_{F\downarrow{_{b}}}$
is the pseudo transformation whose component at any object $(f,a)$
is the 1-cell
$$\begin{array}{c}
(\overset{_\circ}{1}\cdot \bl,1_a):(1_b\circ f,a) \to (f,a),\\
\overset{_\circ}{1}\cdot \bl=\big(\xymatrix@C=15pt{1_b\circ f\ar@2[r]^-{\bl}& f\ar@2[r]^-{\br^{-1}}
            & f\circ 1_{Fa}\ar@2[r]^(.45){1\circ \widehat{F}} &
             f\circ F1_a}\big)\end{array}
$$
 and whose naturality component, at a 1-cell $(\beta,u):(f,a)\to
(f',a')$, is
$$
\xymatrix@C=40pt{(1_b\circ f,a)\ar[r]^{(\overset{_\circ}{1},1)}
\ar[d]_{(1_b\circledcirc \beta,u)}\ar@{}
\ar@{}@<45pt>[d]|(.38){\br^{-1}\cdot\, \bl}|(.51){\cong}
& (f,a)\ar[d]^{(\beta,u)} \\
(1_b\circ f',a')\ar[r]_{(\overset{_\circ}{1},1)}&(f',a');
}
$$

$(\mathbf{D6})$ for any two vertically composable 2-cells
$\xymatrix@C=12pt{p\ar@2[r]^{\sigma}&p'\ar@2[r]^{\tau}&p'' }$ in
$\B$,
 the invertible modification \hspace{0.1cm}${\xi}_{\tau,\sigma}:\tau_*\circ \sigma_*
 \Rrightarrow  (\tau\cdot  \sigma)_*$ has component, at any object
 $(f,a)$, the canonical isomorphism $\bl:1_a\circ 1_a\cong 1_a$
 $$
\xymatrix@C=4pt@R=14pt{
& (p\circ f,a)\ar[ld]_{(\sigma\circledcirc f,1_a)}\ar[rd]^{\ ((\tau\cdot\sigma)\circledcirc f,1_a)}
\ar@{}[d]|(.6){\cong}|(.42){\bl}&
\\
(p'\circ f,a)\ar[rr]_{(\tau\circledcirc f,1_a)}&&(p''\circ f,a);
}
$$

\vspace{0.2cm}$(\mathbf{D7})$ for each 1-cell $p:b\to b'$ of $\B$,
 ${(1_p)}_*=1_{p_*}$, and   \hspace{0.1cm}${\xi}_{_p}$ is the identity modification;

$(\mathbf{D8})$ for every two horizontally composable 2-cells
$\xymatrix@C=0.5pc{b\ar@/^0.5pc/[rr]^{p} \ar@/_0.5pc/[rr]_{q}
\ar@{}[rr]|{\Downarrow\sigma}&  &
b'\ar@/^0.5pc/[rr]^{p'}\ar@/_0.5pc/[rr]_{q'}
\ar@{}[rr]|{\Downarrow\tau}& & b''}$ in $\B$, the equality $
(\tau_*\sigma_*)\!\circ {\chi}{_{p',p}}={\chi}{_{q',q}}\circ
(\tau\circ \sigma)_* $ holds and the  modification
${\chi}_{_{\tau,\sigma}}$ is the identity;

$(\mathbf{D9})$ for every three composable 1-cells
$\xymatrix@C=13pt{b\ar[r]^{p}&b'\ar[r]^{p'}&b''\ar[r]^{p''}&b'''}$
in $\B$, the invertible modification ${\omega}_{_{p'',p',p}}$, at
any object $(f,a)$, is the canonical isomorphism $\br:(1_a\circ
1_a)\circ 1_a\cong 1_a\circ 1_a$,
$$
\xymatrix@C=30pt{
(((p''\circ p')\circ p)\circ f,a)\ar[rr]^{(\aso\circledcirc f,1_a)}
\ar[d]_{(\text{{\aa}},1_a)}
&\ar@{}[d]|(.55){\cong}|(.40){\br}&((p''\circ(p'\circ p))\circ f,a)
\ar[d]^{(\text{{\aa}},1_a)}\\
((p''\circ p')\circ(p\circ f),a)\ar[r]^{(\text{{\aa}},1_a)}&
(p''\circ(p'\circ(p\circ f)),a)
&(p''\circ((p'\circ p)\circ f),a)\ar[l]_{(p''\circledcirc\, \text{{\aa}},1_a)}
;
}
$$

$(\mathbf{D10})$ for any $1$-cell $p:b\to b'$ of $\B$, the
invertible modifications $\gamma_p$ and $\delta_p$, at any object
$(f,a)$ are given by the canonical isomorphism $1_a\circ (1_a\circ
1_a)\cong 1_a$,
$$
\xymatrix@C=30pt{(1_{b'}\circ(p\circ f),a)\ar[r]^-{(\overset{_\circ}{1}\cdot \bl,1_a)}
\ar@{}@<45pt>[d]|(.55){\cong}|(.43){\br\cdot\br} & (p\circ f,a)\ar[d]^-{(\overset{_\circ}{1},1_a)}
\\
((1_{b'}\circ p)\circ f,a)\ar[u]^{(\text{\aa},1_a)}\ar[r]_-{(\bl\circledcirc f,1_a)} & (p\circ f,a)
}
\xymatrix@C=35pt{(p\circ(1_{b'}\circ f),a)\ar[r]^-{(p\circledcirc(\overset{_\circ}{1}\cdot\bl),1_a)}
\ar@{}@<45pt>[d]|(.55){\cong}|(.43){\br\cdot\br} & (p\circ f,a)\ar[d]^{(\overset{_\circ}{1},1_a)}
\\
((p\circ 1_{b'})\circ f,a)\ar[u]^{(\text{\aa},1_a)}\ar[r]_-{(\br\circledcirc f,1_a)} & (p\circ f,a).}
$$

Observe that all the $2$-cells given above are well defined since
all the data is obtained from the constraints of the bicategories
involved and the lax functor $F$. Then the coherence conditions of
these give us the equality (\ref{1102}) in each case. For the same
reason the axioms $(\mathbf{C1})-(\mathbf{C8})$ hold.

{\em Part 2.} In this part, we consider the Grothendieck
construction on the oplax bidiagram of homotopy fibers
$F\!\!\downarrow\,:\B\to \Bicat$, and we shall prove the following:
\begin{lemma}\label{lemQ} There is a homomorphism
\begin{equation}\label{Q}
\xymatrix{Q:\int_{\B}\!F\!\!\downarrow\ \to \A,}
\end{equation}
inducing a homotopy equivalence  on classifying spaces, $\BB
Q:\BB\!\int_{\B}\!F\!\!\downarrow\ \simeq \BB\A$.
\end{lemma}

Before starting the proof of the lemma,  we shall briefly describe
the bicategory $\int_{\B}\!F\!\!\downarrow$. It has objects the
triplets $(f,a,b)$, with $a\in\Ob\A$, $b\in\Ob\B$, and $f:Fa\to b$ a
1-cell of $\B$. Its 1-cells $$(\beta,u,p):(f,a,b)\to (f',a',b'),$$
consist of a $1$-cell $p:b\to b'$ in $\B$, together with a $1$-cell
$(\beta,u):p_*(f,a)=(p\circ f,a)\to (f',a')$ in
$F\!\!\downarrow\!_{b'}$, that is,  a 1-cell $u:a\to a'$ in $\A$ and
a 2-cell $\beta: p\circ f\Rightarrow f'\circ Fu$ in $\B$
$$
\xymatrix{Fa
\ar@{}@<20pt>[d]|(.4){\beta}|(.55){\Rightarrow}
\ar[r]^{Fu}\ar[d]_{f}&Fa'\ar[d]^{f'}\\ b\ar[r]^{p}&b'.}
$$
A 2-cell in $\int_{\B}\!F\!\!\downarrow$,
$$\xymatrix@C=30pt{(f,a,b)\ar@/^1pc/[r]^{(\beta,u,p)}
\ar@/_1pc/[r]_{(\beta',u',p')}
\ar@{}[r]|{\Downarrow(\alpha,\sigma)}& (f',a',b')},
$$
 consists of a 2-cell $\sigma:p\Rightarrow p'$ in $\B$, together with a 2-cell $\alpha:(\beta,u)\Rightarrow (\beta',u')\circ \sigma_*(f,a)$ in $F\!\!\downarrow\!_{b'}$,  that is, (after some work using coherence
equations) a 2-cell
 $\alpha: u\Rightarrow u'\circ 1_a$ in $\A$, such that the equation below
 holds.
$$
\xymatrix@R=35pt@C=35pt{
Fa\ar[r]^{Fu'}\ar[d]_{f}&Fa'\ar[d]^{f'}\ar@{}@<25pt>[d]|{\textstyle =}\\ b
\ar@/^1pc/[r]_{\Uparrow \sigma}\ar@{}@<16pt>[r]|(.4){p'}\ar@{}@<-25pt>[u]|(.6){\Rightarrow}|(.72){\beta'}
\ar[r]_p&b'}\hspace{0.2cm}
\xymatrix@R=35pt@C=35pt{
Fa
\ar@/_1.3pc/[r]\ar@{}@<-19pt>[r]|(.4){Fu}
\ar[r]^{Fu'}_{\Uparrow F(\br\cdot\alpha)}\ar[d]_{f}&Fa'\ar[d]^{f'}\\ b
\ar@{}@<-25pt>[u]|(.25){\Rightarrow}|(.37){\beta} \ar[r]_p&b'}
$$

We shall look carefully at the vertical composition of $2$-cells and
the horizontal composition of $1$-cells in $\int_\B F\!\!\downarrow$
since we will use them later: Given two vertically composable
$2$-cells, say  $(\alpha,\sigma)$ as above and
$(\alpha',\sigma'):(\beta',u',p')\Rightarrow (\beta'',u'',p'')$,
their vertical composition is given by the formula
$$(\alpha',\sigma')\cdot(\alpha,\sigma)=(\alpha'\cdot \br \cdot \alpha,\sigma'\cdot \sigma):
(\beta,u,p)\Rightarrow (\beta'',u'',p'').$$ Given two composable
$1$-cells, say $(\beta,u,p)$ as above and
$(\beta',u',p'):(f',a',b')\to (f'',a'',b'')$, their horizontal
composition is
$$(\beta',u',p')\circ (\beta,u,p)=(F\br^{-1}\cdot(\beta'\circledcirc (1_{p'}\circ
\beta)),(u'\circ u)\circ 1_a,p'\circ p):(f,a,b)\to (f'',a'',b''),$$
where $\beta'\circledcirc (1_{p'}\circ \beta)$ is as in
$(\ref{oo})$, thus
$$
\xymatrix{
Fa
\ar@{}@<-70pt>[d]|(.3){\textstyle F\br^{-1}\cdot(\beta'\circledcirc (1_{p'}\circ
\beta)):}
\ar@/^2pc/[rr]^{F((u'\circ u)\circ 1_a)} \ar@{}@<14pt>[rr]|{\Uparrow
F\br^{-1}\cdot \widehat{F}}
\ar@{}@<24pt>[d]|(.52){\Rightarrow}|(.38){\beta}
\ar[r]^{Fu}\ar[d]_{f}&Fa'
\ar@{}@<24pt>[d]|(.52){\Rightarrow}|(.38){\beta'}
\ar[r]^{Fu'}\ar[d]_{f'}&Fa''
\ar[d]^{f''}\\
b\ar[r]^{p}&b' \ar[r]^{p'}&b''.
}
$$
The identity $1$-cell at an object $(f,a,b)$ is
$$\begin{array}{c}1_{(f,a,b)}=(\overset{_\circ}{1}_{(f,a)}\cdot \bl,1_a,1_b):(f,a,b)\to
(f,a,b).\\[5pt]
\overset{_\circ}{1}_{(f,a)}\cdot
\bl=\Big(\xymatrix@C=10pt{1_b\circ f\ar@{=>}[r]^-{\bl}&
f\ar@2[r]^-{\br^{-1}}& f\circ 1_{Fa}\ar@2[r]^-{1_f\circ
\widehat{F}}&f\circ F(1_a)}\Big)
\end{array}
$$

{\em Proof of Lemma \ref{lemQ}.} The homomorphism $Q$ in $(\ref{Q})$
is defined on cells by
$$\begin{array}{c}\xymatrix@C=30pt{(f,a,b)\ar@/^1pc/[r]^{(\beta,u,p)}
\ar@/_1pc/[r]_{(\beta',u',p')}
\ar@{}[r]|{\Downarrow(\alpha,\sigma)}& (f',a',b')}\ \overset{Q}\mapsto \
\xymatrix{a\ar@/^0.8pc/[r]^{u}\ar@/_0.8pc/[r]_{u'}\ar@{}[r]|{\Downarrow \br\cdot\alpha}&a',}\\[5pt]
\br\cdot\alpha=\big(u\overset{\alpha}\Rightarrow u'\circ 1_a\overset{\br}\Rightarrow
u'\big)
\end{array}
$$
This homomorphism $Q$ is strictly unitary, and its structure
isomorphism at any two composable 1-cells, say $(\beta,u,p)$ as
above and $(\beta',u',p'): (f',a',b')\to (f'',a'',b'')$, is
$$\widehat{Q}=\br_{u'\circ u}:Q((\beta',u',p')\circ (\beta,u,p))\cong Q(\beta',u',p')\circ
Q(\beta,u,p).
$$

To prove that this homomorphism $Q$ induces a homotopy equivalence
on classifying spaces, let us observe that there is also a lax
functor $L:\A\to \int_{\B}\!F\!\!\downarrow$, such that
$Q\,L=1_{\A}$. This is defined on cells of $\A$ by
$$\begin{array}{c}
\xymatrix{a\ar@/^0.8pc/[r]^{u}\ar@/_0.8pc/[r]_{u'}\ar@{}[r]|{\Downarrow\alpha}&a'}
\ \overset{L}\mapsto \
\xymatrix@C=55pt{(1_{Fa},a,Fa)\ar@/^1pc/[r]^{(\bl^{-1}\cdot \br,u,Fu)}
\ar@/_1pc/[r]_{(\bl^{-1}\cdot \br,u',Fu')}
\ar@{}[r]|{\Downarrow(\br^{-1}\cdot\alpha,F\alpha)}& (1_{Fa'},a',Fa'),}\\[5pt]
\br^{-1}\cdot \alpha=\big(u\overset{\alpha}\Rightarrow u'\overset{\br^{-1}}\Rightarrow
u'\circ 1_a\big)
\end{array}
$$
where the first component of $(\bl^{-1}\cdot \br,u,Fu)$ is the
canonical isomorphism $Fu\circ 1_{Fa}\cong 1_{Fa'}\circ Fu$. Its
structure 2-cells, at any pair of composable 1-cells
$a\overset{u}\to a'\overset{u'}\to a''$ and at any object $a$ of
$\A$, are respectively defined by
$$
\begin{array}{ll}
\widehat{L}_{u',u}=(1_{(u'\circ u)\circ 1_a},\widehat{F}_{u',u}):\ Lu'\circ Lu \Rightarrow L(u'\circ u),\\
\widehat{L}_{a}=(\br^{-1}_{1_a},\widehat{F}_{a}): 1_{La} \Rightarrow L1_{a}.
\end{array}
$$

The equality  $QL=1_{\A}$ is easily checked. Furthermore, there is
an oplax transformation $\iota:LQ\Rightarrow
1_{\int_{\B}\!F\downarrow}$ assigning to each object $(f,a,b)$ of
the bicategory $\int_{\B}\!F\!\!\downarrow$ the 1-cell
$$
\iota(f,a,b)=(1_f\circ \widehat{F}_a,1_a,f):(1_{Fa},a,Fa)\to (f,a,b),
$$
and whose naturality component at any 1-cell $(\beta,u,p):(f,a,b)\to
(f',a',b')$ is the 2-cell
$$
\xymatrix@C=45pt{
(1_{Fa},a,Fa)
\ar@{}@<60pt>[d]|(.35){\widehat{\iota}\,=((\bl^{-1}\circ 1)\circ 1,\beta)}|(.5){\Rightarrow}
\ar[r]^{(\bl^{-1}\cdot \br,u,Fu)}\ar[d]_{(1_f\circ
\widehat{F}_a,1_a,f)}&
(1_{Fa'},a',Fa')\ar[d]^{(1_{f'}\circ \widehat{F}_{a'},1_{a'},f')}\\
(f,a,b)\ar[r]_-{(\beta,u,p)}&(f',a',b').
}
$$

Therefore, by taking classifying spaces, we have $\BB Q\,\BB
L=1_{\BB\A}$ and, by Lemma \ref{trans}, $\BB L\,\BB Q\simeq
1_{\BB\int_{\B}\!F\downarrow}$, whence $\BB Q$ is actually a
homotopy equivalence. \qed

{\em Part 3.} We complete here the proof of the theorem as follows:
There is a canonical homomorphism
\begin{equation}\label{fba2}\xymatrix{\bar{F}:\int_\B F\!\!\downarrow \,\longrightarrow \,
\int_\B \B\!\!\downarrow}
\end{equation}
making commutative,  for any object $b\in\Ob\B$, the diagrams
$$
\begin{array}{cc}
\xymatrix{\ar@{}@<-20pt>[d]|(.3){\textstyle (A):}&&\\F\!\!\downarrow \!b \ar[r]^{J} \ar[d]_{\bar{F}} \ar@/^1.5pc/[rr]^{P} &
      \int_\B F\!\!\downarrow \ar[r]^{Q} \ar@{-->}[d]^{\bar{F}} &
      \A \ar[d]^{F}
      \\
      \B\!\!\downarrow \!b \ar[r]^{J} \ar@/_1.5pc/[rr]_{P} &
      \int_\B \B\!\!\downarrow \ar[r]^{Q} &
        \B
}
\hspace{1cm}
\xymatrix{\ar@{}@<-20pt>[d]|(.3){\textstyle (B):}&&\\ F\!\!\downarrow\! b \ar[r]^{F} \ar[d]_{J} \ar@/^1.5pc/[rr] &
      \B\!\!\downarrow\! b \ar[r] \ar[d]^J &
      [0]\ar[d]^{b}
      \\
      \int_\B F\!\!\downarrow \ar@{-->}[r]^{\bar{F}} \ar@/_1.5pc/[rr]_P &
      \int_\B \B\!\!\downarrow \ar[r]^{P} &
      \B
}
\end{array}
$$
in which  $Q:\int_\B F\!\!\downarrow \to \A$ is the homomorphism in
$(\ref{Q})$ and $Q:\int_\B \B\!\!\downarrow \to \B$ is the
corresponding one for $F=1_{\B}$,  all the 2-functors $P$ are  the
canonical projections $(\ref{proj2fun})$, and the embedding
homomorphisms $J$ are the corresponding ones defined as in
$(\ref{emj})$. This homomorphism $(\ref{fba2})$ is  defined on cells
by
$$\begin{array}{c}
\xymatrix@C=30pt{(f,a,b)\ar@/^1pc/[r]^{(\beta,u,p)}
\ar@/_1pc/[r]_{(\beta',u',p')}
\ar@{}[r]|{\Downarrow(\alpha,\sigma)}& (f',a',b')}\,\overset{\bar{F}}\mapsto\,
\xymatrix@C=30pt{(f, F a,b)\ar@/^1pc/[rr]^{(\beta, F u,p)}
\ar@/_1pc/[rr]_{(\beta', F u',p')}
\ar@{}[rr]|{\Downarrow( \br^{-1}\cdot F\br\cdot F\alpha,\sigma)}&& (f', F a',b').}
\\
\br^{-1}\cdot F\br\cdot F\alpha=\Big(\xymatrix@C=15pt{Fu\ar@2[r]^-{F\alpha}&
 F(u'\circ 1_a)\ar@2[r]^-{F\br}& Fu'\ar@2[r]^-{\br^{-1}} & Fu'\circ 1_{Fa}}\Big)
 \end{array}
$$
Its composition constraint at a pair of composable $1$-cells, say
$(\beta,u,p)$  as above and $(\beta',u',p'):(f',a',b')\to
(f'',a'',b'')$, is the 2-cell
$$\begin{array}{c}
(\widetilde{F},1_{p'\circ p}):\bar{F}(\beta',u',p')\circ
\bar{F}(\beta,u,p)\Rightarrow \bar{F}((\beta',u',p')\circ (\beta,u,p)),\\
\widetilde{F}=\Big(\xymatrix@C=30pt{(Fu'\circ Fu)\circ 1_{Fa}\ar@2[r]^{\widehat{F}\circ 1}
&F(u'\circ u)\circ 1_{Fa}\ar@2[r]^-{F(\br^{-1})\circ 1} &F((u'\circ u)\circ 1_a)\circ 1_{Fa}}\Big)
\end{array}
$$
while its unit constraint at an object $(f,a,b)$ is
$$\begin{array}{c}
(\widetilde{F},1_{1_b}):1_{\bar{F}(f,a,b)}\Rightarrow
\bar{F}(1_{(f,a,b)}).\\
\widetilde{F}=\Big(\xymatrix@C=15pt{1_{Fa}\ar@2[r]^-{\br^{-1}}&1_{Fa}\circ
1_{Fa}\ar@2[r]^-{\widehat{F}\circ 1} &F(1_a)\circ
1_{Fa}}\Big)\end{array}
$$

Let us now observe that (the covariant and oplax version of) Theorem
\ref{th1} applies both to the bidiagram of homotopy fibres
$F\!\!\downarrow$, by hypothesis, and to the bidiagram of comma
bicategories $\B\!\!\downarrow$, since the spaces
$\BB\B\!\!\downarrow{\!_{b}}$ are contractible by Lemma \ref{cont}
and therefore any 1-cell $p:b\to b'$ in $\B$ obviously induces a
homotopy equivalence $\BB p_*:\BB\B\!\!\downarrow{\!_{b}}\simeq
\BB\B\!\!\downarrow{\!_{b'}}$. Hence, the squares
\begin{equation}\label{squbicat}
\begin{array}{c}\xymatrix{
F\!\!\downarrow{\!_{b}}\ar[r]^{J}\ar[d]&\int_\B F\!\!\downarrow\ar[d]^{P}\\ [0]\ar[r]^{b}&\B
}  \hspace{0.5cm}
\xymatrix{
\B\!\!\downarrow{\!_{b}}\ar[r]^{J}\ar[d]&\int_\B \B\!\!\downarrow\ar[d]^{P}\\ [0]\ar[r]^{b}&\B
}
\end{array}
\end{equation}
induce homotopy cartesian squares on classifying spaces
$$
\begin{array}{c}\xymatrix{
\BB F\!\!\downarrow{\!_{b}}\ar[r]^{\BB J}\ar[d]&\BB\int_\B F\!\!\downarrow\ar[d]^{\BB P}\\
pt\ar[r]^{\BB b}&\BB\B,
}  \hspace{0.5cm}
\xymatrix{
\BB\B\!\!\downarrow{\!_{b}}\ar[r]^{\BB J}\ar[d]&\BB\int_\B \B\!\!\downarrow\ar[d]^{\BB P}\\
pt\ar[r]^{\BB b}&\BB \B.
}
\end{array}
$$

By \cite[II, Lemma 8.22 (2)(b)]{g-j}, it follows from the
commutativity of diagram $(B)$ above that the induced square
$$
\xymatrix{\BB F\!\!\downarrow{\!_{b}}\ar[r]^{\BB\bar{F}}\ar[d]_{\BB J}&\BB \B\!\!\downarrow{\!_{b}}\ar[d]^{\BB J}\\
\BB\int_\B F\!\!\downarrow\ar[r]^{\BB\bar{F}}&\BB \int_B \B\!\!\downarrow}
$$
is homotopy cartesian. Then, by \cite[II, Lemma 8.22 (1),
(2)(a)]{g-j}, the theorem follows from the commutativity of diagram
$(A)$, since, by Lemma \ref{lemQ}, in the induced square
$$
\xymatrix{\BB \int_B F\!\!\downarrow\ar[r]^{\BB\bar{F}}\ar[d]_{\BB Q}&\BB \int_\B\B\!\!\downarrow\ar[d]^{\BB Q}\\
\BB\A\ar[r]^{\BB F}&\BB \B}
$$
both maps $\BB Q$ are homotopy equivalences and therefore it is
homotopy cartesian.
\end{proof}

The following corollary generalizes Quillen's Theorem A in
\cite{quillen}:
\begin{theorem} \label{A}
Let $F:\A\to \B$ be a lax functor between bicategories. The induced
map on classifying spaces $\BB F:\BB \A\to \BB\B$ is a homotopy
equivalence whenever the classifying spaces of the homotopy fiber
bicategories $\BB F\!\!\downarrow{\!_{b}}$ are contractible for all
objects $b$ of $\B$.
\end{theorem}

Particular cases of the result above have been also stated in
\cite[Theorem 1.2]{b-c},  for the case when $F:\A\to \B$ is any
2-functor between 2-categories, and in \cite[Theorem 6.4]{Hoyo}, for the
case when $F$ is a lax functor from a category $\A$ to a 2-category
$\B$. In \cite[Th\'{e}or\`{e}me 6.5]{chiche}, it is stated a relative Theorem A for lax functors between 2-categories, which also implies the particular case of Theorem \ref{A} when $F$ is any lax functor between 2-categories.

\begin{example}{\em
Let $(\mathcal{M},\otimes)=(\mathcal{M},\otimes,I,\aso,\bl,\br)$ be
a monoidal category (see e.g. \cite{maclane}), and let $\Sigma
(\mathcal{M},\otimes)$ denote its {\em suspension} or {\em
delooping} bicategory. That is, $\Sigma (\mathcal{M},\otimes)$ is
the  bicategory with only one object, say $\star$, whose
hom-category is $\mathcal{M}$, and whose horizontal composition is
given by the tensor functor $\otimes:\mathcal{M}\times
\mathcal{M}\to \mathcal{M}$.  The identity 1-cell on the object is
the unit object $I$ of the monoidal category, and the constraints
$\aso$, $\bl$, and $\br$ for $\Sigma (\mathcal{M},\otimes)$ are just
those of the monoidal category.

By \cite[Theorem 1]{b-c},
$$
\BB(\mathcal{M},\otimes)=\BB\Sigma (\mathcal{M},\otimes),$$ that is,
the classifying space of the monoidal category is the  classifying
space of its suspension bicategory. Then, Theorem \ref{B} is
applicable to monoidal functors between monoidal categories.

However, we should stress that the homotopy fiber bicategory of the
homomorphism between the suspension bicategories that a monoidal
functor $F:(\mathcal{M},\otimes)\to (\mathcal{M}',\otimes)$ defines,
 $\Sigma F:\Sigma(\mathcal{M},\otimes)\to \Sigma(\mathcal{M}',\otimes)$, at the
unique object of $\Sigma(\mathcal{M}',\otimes)$, is not a monoidal
category but a genuine bicategory: The 0-cells of $\Sigma
F\!\!\downarrow_\star$ are the objects $x'\in\mathcal{M}'$, its
1-cells  $(u',x):x'\to y'$ are pairs with $x$ an object in
$\mathcal{M}$ and $u':x'\to y'\otimes F(x)$ a morphism in
$\mathcal{M}'$, and its 2-cells $$ \xymatrix @C=8pt{x'
\ar@/^0.7pc/[rr]^{(u',x)} \ar@/_0.7pc/[rr]_-{
(v',y)}\ar@{}[rr]|{\Downarrow\! u } &  & y'}
$$
are those morphisms $u:x\to y$ in $\mathcal{M}$ making commutative the triangle $$
\xymatrix{x'\ar[r]^{u'}\ar[rd]_{v'}&y'\otimes Fx\ar[d]^{y'\otimes
Fu}\\ & y'\otimes Fy. }
$$
The vertical composition of 2-cells is given by the composition of arrows in $\mathcal{M}$. The horizontal composition of two 1-cells $\xymatrix{x'\ar[r]^{(u',x)}&y'\ar[r]^{(v',y)}&z'}$ is the 1-cell
$
(v' \circledcirc u', y\otimes x): x'\to z'
$,
$$
v' \circledcirc u'=\Big(\xymatrix{x'\ar[r]^-{u'}&y'\otimes
Fx\ar[r]^-{v'\otimes Fx}&(z'\otimes Fy)\otimes Fx}\cong  z'\otimes
(Fy\otimes Fx)\cong z'\otimes
 F(y\otimes x)\Big)
$$
and the horizontal composition of 2-cells is given by tensor product of arrows in $\mathcal{M}$. The identity 1-cell of any 0-cell $x$ is $ (\overset{_\circ}{1}_x,I):x\to x$, where
$
\overset{_\circ}{1}_x=(x'\cong x'\otimes I'\cong x'\otimes FI)
$. The associativity, left and right constraints are obtained from those of $(\mathcal{M},\otimes)$
by the formulas
$$
\boldsymbol{a}_{(w',z),(v',y),(u',x)}=\boldsymbol{a}_{z,y,x},\hspace{0.3cm}
\boldsymbol{r}_{(u',x)}= \boldsymbol{r}_x,\hspace{0.3cm}
\boldsymbol{l}_{(u',x)}=\boldsymbol{l}_x.
$$
Following the terminology of \cite[page 228]{b-c2}, we shall call this
bicategory $\Sigma F\!\!\downarrow_\star$ the {\em homotopy fiber bicategory} of the monoidal functor $F:(\mathcal{M},\otimes)\to (\mathcal{M}',\otimes)$, and write it by $\mathcal{K}_F$.

Every object $z'$ of $\mathcal{M}'$, determines a 2-endofunctor
$z'\otimes -:\mathcal{K}_F\to \mathcal{K}_F$, which is defined on
cells by
 $$ \xymatrix @C=8pt{x'
\ar@/^0.7pc/[rr]^{(u',x)} \ar@/_0.7pc/[rr]_-{
(v',y)}\ar@{}[rr]|{\Downarrow\! u } &  & y'}
 \mapsto \
\xymatrix @C=10pt{z'\otimes x'
\ar@/^0.7pc/[rr]^{(z'\circledcirc u',x)} \ar@/_0.7pc/[rr]_-{
(z'\circledcirc v',y)}\ar@{}[rr]|{\Downarrow\! u } &  & z'\otimes y',}
$$
where  $z'\circledcirc u'=\Big(\xymatrix{z'\otimes
x'\ar[r]^-{z'\otimes u'}&z'\otimes (y'\otimes Fx)}\cong (z'\otimes
y')\otimes Fx\Big)$, and
 from  Theorems \ref{B} and \ref{A}, we get the following:

\begin{theorem} For any monoidal functor $F:(\mathcal{M},\otimes)\to
(\mathcal{M}',\otimes)$, the following statements hold:

(i) There is an induced homotopy fiber sequence
$$\BB\mathcal{K}_F\to\BB(\mathcal{M},\otimes)\overset{\BB F}\longrightarrow \BB(\mathcal{M}',\otimes),$$
whenever the induced maps $\BB(z'\otimes
-):\BB\mathcal{K}_F\to\BB\mathcal{K}_F$ are homotopy
autoequivalences, for all $z'\in \mbox{Ob}\mathcal{M}'$.

(ii) The induced map $\BB F: \BB(\mathcal{M},\otimes)\to
\BB(\mathcal{M}',\otimes)$ is a homotopy equivalence if the space
$\BB\mathcal{K}_F$ is contractible.
\end{theorem}

For any monoidal category $(\mathcal{M},\otimes)$, pseudo bidiagrams
of categories over its suspension bicategory,
$$
\mathcal{N}=(\mathcal{N},\chi):\Sigma(\mathcal{M},\otimes)^{\mathrm{op}}\to \Cat,
$$
are interesting to consider, since they can be regarded as a
category $\mathcal{N}$ (the one associated to the unique object of
the suspension bicategory) endowed with a coherente right pseudo
action of the monoidal category $(\mathcal{M},\otimes)$ (see e.g.
\cite[\S 1]{jardine}). Namely, by the functor
$\otimes:\mathcal{N}\times \mathcal{M}\to \mathcal{N}$, which is
defined on objects by $a\otimes x=x^*a$ and on morphism by
$$ (a\overset{f}\to b)\otimes
(x\overset{u}\to
y)=\xymatrix@C=20pt{\big(x^*a\ar[r]^-{x^*\!f}&x^*b\ar[r]^-{u^*b}&y^*b\big)}=
\xymatrix@C=20pt{\big(x^*a\ar[r]^-{u^*\!a}&y^*a\ar[r]^-{y^*f}&y^*b\big),}
$$
together with the coherent natural isomorphisms
$$\begin{array}{l}
\xymatrix{(a\otimes x)\otimes y=y^*x^*a\ar[r]^-{\chi_{x,y}a}&(x\otimes y)^*a=a\otimes (x\otimes y)}\\
\  \xymatrix{a\ar[r]^-{\chi_{_I}a}&I^*a=a\otimes I.}
\end{array}
$$

For each such $(\mathcal{M},\otimes)$-category $\mathcal{N}$, the
cells of the bicategory
$\int_{\Sigma(\mathcal{M},\otimes)}\mathcal{N}$ has the following
easy description: Its objects are the same as the objects of the
category $\mathcal{N}$. A 1-cell $(f,x):a\to b$  is a pair with $x$
an object of $\mathcal{M}$ and $f:a\to b\otimes x$ a morphism in
$\mathcal{N}$, and a 2-cell $$\xymatrix @C=6pt {a
\ar@/^0.7pc/[rr]^{(f,x)} \ar@/_0.7pc/[rr]_{ (g,y)}
\ar@{}[rr]|{\Downarrow\!u}& {} &b }$$ is a morphism $u:x\to y$ in
$\mathcal{M}$ such that the triangle
$$
\xymatrix@C=12pt@R=16pt{&a\ar[rd]^{g}\ar[ld]_{f}&\\ b\otimes x \ar[rr]^{b \otimes u}&&b\otimes y.}
$$
is commutative. Many of the homotopy theoretical properties of the
classifying space of the monoidal category,
$\BB(\mathcal{M},\otimes)$, can actually be more easily reviewed by
using Grothendieck bicategories
$\int_{\Sigma(\mathcal{M},\otimes)}\mathcal{N}$,  instead of the
Borel pseudo simplicial categories
$$E_{(\mathcal{M},\otimes)}\mathcal{N}:\Delta^{\!{\mathrm{op}}}\to
\Cat, \hspace{0.3cm}[p]\mapsto \mathcal{N}\times \mathcal{M}^p
$$as, for example,
Jardine did in \cite{jardine} for $(\mathcal{M},\otimes)$-categories
$\mathcal{N}$.

Thus, one sees, for example, that if the action is such that
multiplication by each object $x$ of $\mathcal{M}$, that is, the
endofunctor $-\otimes x:\mathcal{N}\to \mathcal{N}$, induces a
homotopy equivalence $\BB\mathcal{N}\simeq \BB\mathcal{N}$, then, by
Theorem \ref{th1}, one has an induced homotopy fiber sequence (cf.
\cite[Proposition 3.5]{jardine})
$$\xymatrix{\BB \mathcal{N}\to\BB \int_{\Sigma(\mathcal{M},\otimes)}\mathcal{N} \,\overset{\BB P}\longrightarrow\, \BB(\mathcal{M},\otimes).}$$

In particular, the right action of $(\mathcal{M},\otimes)$ on the
underlying category $\mathcal{M}$ leads to the bicategory
$$\xymatrix{\int_{\Sigma(\mathcal{M},\otimes)}\mathcal{M}=
\Sigma(\mathcal{M},\otimes)\!\!\downarrow_{\star}},
$$ the comma bicategory of the suspension bicategory over its unique object, whose
classifying space is contractible by Lemma \ref{cont} (cf.
\cite[Proposition 3.8]{jardine}). Then, it follows the well-known
result by Mac Lane \cite{mac} and Stasheff \cite{sta}  that there is
a homotopy equivalence
$$\BB \mathcal{M}\simeq \Omega\BB(\mathcal{M},\otimes),$$ between the
classifying space of the underlying category and the loop space of
the classifying space of the monoidal category, whenever
multiplication by each object $x\in \Ob\mathcal{M}$, $y\mapsto
y\otimes x$, induces a homotopy autoequivalence on $\BB\mathcal{M}$
(cf. Example \ref{ombi}). }
\end{example}

\end{document}